\pdfoutput=1
\RequirePackage{ifpdf}
\ifpdf 
\documentclass[pdftex]{sigma}
\else
\documentclass{sigma}
\fi

\usepackage{texdraw}
\usepackage[all]{xy}
\usepackage{pb-diagram}

\numberwithin{equation}{section}

\newtheorem{Theorem}{Theorem}[section]
\newtheorem*{Theorem*}{Theorem}
\newtheorem{Corollary}[Theorem]{Corollary}
\newtheorem{Lemma}[Theorem]{Lemma}
\newtheorem{Proposition}[Theorem]{Proposition}

\newtheorem{theoremx}[Theorem]{Theorem}
{ \theoremstyle{definition}

 \newtheorem{Remark}[Theorem]{Remark}
 \newtheorem{observation}[Theorem]{Observation}
}

\def\BN{\mathbb N}
\def\BZ{\mathbb Z}
\def\BQ{\mathbb Q}
\def\BR{\mathbb R}
\def\BP{\mathbb P}
\def\BC{\mathbb C}

\def\calE{\mathcal E}
\def\calI{\mathcal I}

\def\calD{\mathcal D}
\def\calS{\mathcal S}
\def\P{\mathbb P}
\def\s{\sigma}

\def\la{\langle}
\def\ra{\rangle}

\def\ti{\widetilde}
\def\SL{\mathrm{SL}}
\def\PSL{\mathrm{PSL}}
\def\GL{\mathrm{GL}}

\def\ID{I_{\Delta}}
\def\Ind{\mathrm{Ind}}

\def\tq{\tilde{q}}
\def\Res{\mathrm{Res}}

\def\VC{\mathrm{Vol}_\BC}
\def\a{\alpha}
\def\z{\zeta}

\def\+{ + }
\def\m{ - }

\def\be{\begin{equation}}
\def\ee{\end{equation}}
\def\h{\frac12}
\def\sm{\setminus}
\def\Phih{\widehat\Phi}
\def\Psih{\widehat\Psi}
\def\ve{\varepsilon}

\def\tv{\tilde\varepsilon}
\def\Li{\mathrm{Li}}

\def\ps#1{(#1;q)_\infty}
\def\m{ - }
\def\R{\mathbb R}
\def\C{\mathbb C}
\def\Z{\mathbb Z}

\def\Q{\mathbb Q}
\def\ct{\mathrm{c.t.}}
\def\Th{\Theta_q}
\def\t{\tau}
\def\ttau{\tilde\t}
\def\th{\theta}
\def\O{{\rm O}}

\def\J{\mathbf J}
\def\sma#1#2#3#4{\bigl(\smallmatrix#1&#2\\#3&#4\endsmallmatrix\bigr)}

\def\Hab{{\mathrm{Hab}}} 

\def\diag{\operatorname{diag}}
\def\Re{\operatorname{Re}}
\def\Im{\operatorname{Im}}
\newenvironment{psmall}{\left(\begin{smallmatrix}}{\end{smallmatrix}\right)}

\def\Qbar{\overline{\BQ}}
\def\bg{\overline\g}
\def\g{\gamma}

\def\den{\text{den}}

\def\i{^{-1}}

\def\bb{\mathsf{b}}
\def\bJ{\mathbf{J}}

\def\E{\mathcal E}
\def\wh{\widehat}
\def\Tsim{\;\stackrel \cdot =\;}
\def\thin{\hskip1pt}

\begin{document}
\allowdisplaybreaks

\newcommand{\arXivNumber}{2304.09377}

\renewcommand{\PaperNumber}{082}

\FirstPageHeading

\ShortArticleName{Knots and Their Related $q$-Series}

\ArticleName{Knots and Their Related $\boldsymbol{q}$-Series}

\Author{Stavros GAROUFALIDIS~$^{\rm a}$ and Don ZAGIER~$^{\rm bc}$}

\AuthorNameForHeading{S.~Garoufalidis and D.~Zagier}

\Address{$^{\rm a)}$~International Center for Mathematics, Department of Mathematics,\\
\hphantom{$^{\rm a)}$}~Southern University of Science and Technology, Shenzhen, P.R.~China}
\EmailD{\href{mailto:stavros@mpim-bonn.mpg.de}{stavros@mpim-bonn.mpg.de}}
\URLaddressD{\url{http://people.mpim-bonn.mpg.de/stavros}}

\Address{$^{\rm b)}$~Max Planck Institute for Mathematics, Bonn, Germany}
\EmailD{\href{mailto:dbz@mpim-bonn.mpg.de}{dbz@mpim-bonn.mpg.de}}
\URLaddressD{\url{http://people.mpim-bonn.mpg.de/zagier}}

\Address{$^{\rm c)}$~International Centre for Theoretical Physics, Trieste, Italy}

\ArticleDates{Received April 25, 2023, in final form October 17, 2023; Published online November 01, 2023}

\Abstract{We discuss a matrix of periodic holomorphic functions in the upper and lower half-plane which can be obtained from a factorization of an Andersen--Kashaev state integral of a knot complement with remarkable analytic and asymptotic properties that defines a~${\rm PSL}_2({\mathbb Z})$-cocycle on the space of matrix-valued piecewise analytic functions on the real numbers. We identify the corresponding cocycle with the one coming from the Kashaev invariant of a knot (and its matrix-valued extension) via the refined quantum modularity conjecture of~[arXiv:2111.06645] and also relate the matrix-valued invariant with the 3D-index of Dimofte--Gaiotto--Gukov. The cocycle also has an analytic extendability property that leads to the notion of a matrix-valued holomorphic quantum modular form. This is a~tale of several independent discoveries, both empirical and theoretical, all illustrated by the three simplest hyperbolic knots.}


\Classification{57N10; 57K16; 57K14; 57K10}

\section{Introduction}

In this paper, which is a companion of~\cite{GZ:kashaev}, we want to tell a story
about $q$-series and quantum invariants of knots that seems to us very interesting.
The story started 11 years ago with the challenge to compute the asymptotic expansion
at $q \to 1$ of a $q$-hypergeometric series that appeared in the evaluation of a
tetrahedron quantum spin network. As it turned out, when $q={\rm e}^{2\pi {\rm i} \tau}$ with
$\tau$ tending to zero on the positive imaginary axis, the asymptotics were
oscillatory (with approximate oscillation $0.32306$), and after some experimentation,
it was found that the oscillation was given by the volume of the simplest hyperbolic
(figure eight) knot, divided by~$2 \pi$. The appearance of the $4_1$ knot was a bit
strange, since this knot has little to do with the tetrahedral spin network (or its
complement) in Euclidean or hyperbolic 3-dimensional space. This strange
coincidence persisted further, where it was found by a numerical computation that
the first and the second terms in the asymptotic expansion were, after some minor
normalization, rational numbers with numerator 11 and 697, respectively. A search
in our databases revealed that the number 697 appears as the second coefficient
in the asymptotic expansion of the $4_1$ knot, whereas the number 11 appears as the
first coefficient. This was surely not an accident! Using numerical methods, we
were able to match the asymptotics of the Kashaev invariant of the $4_1$ knot to the
radial asymptotics of the above $q$-series to over 100 terms.

So, our $q$-series was certainly attached to an invariant of the $4_1$-knot. A
systematic collection of such knot invariants (indexed by a pair of integers) was
given by the 3D-index of Dimofte--Gaiotto--Gukov~\cite{DGG2, DGG1}, and in fact,
our $q$-series could be re-written as a $q$-hypergeometric sum~$G_0(q)$ related to
the 3D-index, and nearly, but not quite, matched to the so-called total 3D-index.
An illegitimate (i.e., formal, but divergent) computation of the total 3D-index
suggested that the latter should equal to $G_0(q)^2$, but a computation showed that
it did not agree. Further attempts to identify the quotient of the total 3D-index
by $G_0(q)$ did not produce any results.

The next source of $q$-series attached to knots was the state-integral of
Andersen--Kashaev~\cite{AK}. Although the latter is an analytic function of $\tau$
in the cut plane $\BC'=\BC\sm(-\infty,0]$,
it was well known in the physics literature (see~\cite{Beem})
that it should factorize into a finite sum of products of $q$-series times $\tq$-series,
where $q={\rm e}^{2 \pi {\rm i} \tau}$ and $\tq={\rm e}^{-2\pi {\rm i} /\tau}$. In fact, Kashaev and the first
author exactly did so for the state-integral of the $4_1$ knot (and for one dimensional
state-integrals in general) and found out a second $q$-series
$G_1(q)$~\cite{GK:qseries}. What is more, the total 3D-index of the $4_1$ knot
experimentally was checked to be the product $G_0(q) G_1(q)$, a statement that can be
proven rigorously.

We next looked at asymptotics of the vector $(G_0(q),G_1(q))$ of $q$-series of the
$4_1$ knot when $q$ approaches a root of unity ${\rm e}^{2\pi {\rm i} \a}$ (for a rational number
$\a$), and without a surprise this time, we found the pair of asymptotic series
$\wh\Phi^{(\s_1)}_\a(2\pi {\rm i}\tau)$ and $\wh\Phi^{(\s_2)}_\a(2\pi {\rm i}\tau)$ (corresponding
to the geometric representation of the $4_1$ knot and its complex conjugate) that
appear in a refinement of the quantum modularity conjecture~\cite{GZ:kashaev}.
Replacing the~$q$ and $\tq$-series in the state integral when~$q$ is near a root of
unity by their asymptotic expansions produced a bilinear combination of factorially
divergent series which are convergent power series! This phenomenon was illustrated
by a~dramatic drop in the growth rate of the 150-th coefficient of the corresponding
power series.

Having understood the story for the simplest hyperbolic knot, we observed two new
phenomena. One is quadratic relations (which are trivial for the $4_1$ knot)
for the vector of 3 $q$-series (inside and outside the unit disk) for the $5_2$ knot,
and for the vector of $6$ $q$-series for the $(-2,3,7)$-pretzel knot. Another is
the presence of a level, being $2$ for the $(-2,3,7)$-pretzel knot, presumably related
to the fact that its Newton polygon has half-integer slopes.

Returning to the case of the $4_1$ knot, the factorization of its state-integral
suggested that we look at a bilinear $q$ and $\tq$-combination of the vector
$(G_0(q),G_1(q))$
of $q$-series where now $\tq={\rm e}^{2\pi {\rm i} \gamma(\tau)}$ for a fixed element $\gamma$ of
$\SL_2(\BZ)$ (the case of the original state-integral being the one with
$\gamma=\sma 0{-1}1{\hphantom{-}0}$). A priori, this function is analytic only for
$\tau \in \BC\sm\BR$, but a numerical computation revealed that this function
is analytic on a cut plane $\BC_\gamma$. This suggested an extension of the
Andersen--Kashaev state integral that depends on an element $\gamma$ of $\SL_2(\BZ)$,
and even more to an $\SL_2(\BZ)$-version of the Faddeev quantum dilogarithm, which
is studied in current joint work of Kashaev and the authors~\cite{GKZ:modular}.

A closer look at the asymptotics of the vector $(G_0(q),G_1(q))$ as $q$ approaches $1$,
shows that they were given by linear combinations of a pair of asymptotic series
$\wh\Phi^{(\s_1)}(2\pi {\rm i}\tau)$ and $\wh\Phi^{(\s_2)}(2\pi {\rm i}\tau)$.
This suggested that suitable linear combination of the vector $(G_0(q),G_1(q))$
should be simply asymptotic to one of the two $\wh\Phi^{(\s_j)}(2\pi {\rm i}\tau)$ series
above. However, this statement is incorrect.
Instead, the radial asymptotics when $q={\rm e}^{2\pi {\rm i} \tau}$ and $\tau$ tends to zero in a
fixed ray $\arg(\tau)=\th_0$ depend on the ray, but different rays detect asymptotic
expansions of the form ${\rm e}^{-2\pi {\rm i} m/\tau}\wh\Phi^{(\s_j)}(2\pi {\rm i}\tau)$ for $m$ a
nonnegative integer. When $\arg(\tau)=\pi/2$, these exponentially small corrections
cannot be numerically observed, however when $\arg(\tau)$ is near $0$ or $\pi$, one
can indeed see a multiple of these series ${\rm e}^{-2\pi {\rm i} m/\tau}\wh\Phi^{(\s_j)}(2\pi {\rm i}\tau)$,
appearing, and what is more, the multiple is an integer number. This phenomenon
is already hinted by the bilinear factorization of the state-integral as a finite sum
of products of $q$-series times $\tq$-series, and was glimpsed in the present work,
and studied more extensively in the work of Gu--Mari\~no and the first
author~\cite{GGM,GGM:peacock}.
This lead to a~matrix \smash[t]{$\sma {G_0^{(m)}(q)}{G_1^{(m)}(q)}{G_0^{(m+1)}(q)}{G_1^{(m+1)}(q)}$}
whose entries are descendant $q$-series indexed by the integers with
$G_0^{(0)}(q) = G_0(q)$ and $G_1^{(1)}(q)=G_1(q)$.

The matrix of descendant $q$-series defined for $|q| \neq 1$ lead to a matrix of
asymptotic series, and to a matrix-valued $\PSL_2(\BZ)$-cocycle whose value at
$\sma 0{-1}1{\hphantom{-}0}$ is given by a matrix of descendant Andersen--Kashaev state-integrals
and whose value at $\gamma \in \SL_2(\BZ)$ is given by the matrix of descendant
state-integral invariants of~\cite{GKZ:modular}.

The two matrix-valued cocycles, one from~\cite{GZ:kashaev} and the other
one from the current paper agree at the rational numbers. This follows from
a second factorization property of the state-integrals at rational
points~\cite{GK:evaluation}. This leads to the notion of a
\emph{holomorphic quantum modular form}, a generalization of a mock modular form,
whose realization as periodic functions at rational numbers was the focus
of~\cite{GZ:kashaev} and whose realization as periodic holomorphic
functions in $\BC\sm\BR$ was the focus of our paper.

In this paper, we will have a number of statements called ``Observations'', all of which
were first observed empirically, but of which some are now proved and others still
conjectural. We will indicate this individually in each case.

A preliminary draft of this paper was already written in 2012 but then not
published because we kept finding new results which made the older versions
obsolete. In the present paper, the relation to the perturbative series
and functions on roots of unity treated in~\cite{GZ:kashaev} have finally become
clear. Related aspects of this work appeared in~\cite{DG,DG2,Ga:arbeit,GK:qseries,Gukov:resurgence,
Gukov:BPS}. Modular linear $q$-difference equations
were introduced in~\cite{GW:modular}. An extension of the matrix-valued $q$-series
to a~matrix of one additional row and column that sees the trivial
$\PSL_2(\BC)$-representation was given in~\cite{GGMW:trivial}. A detailed study of
the asymptotics of the full 3D-index (as opposed to its total version discussed
here) and of the related Turaev--Viro invariant was given in~\cite{GW:asy3D}.
A detailed study of the $6 \times 6$ matrix of $q$-series associated to the
$(-2,3,7)$-pretzel knot is given by Ni~An and Yunsheng~Li in~\cite{237}.

Finally, we mention that this story of quantum knot invariants (i.e., 3-mani\-folds
with torus boundary) extends to the case of the Witten--Reshetikhin--Turaev
of closed hyperbolic 3-mani\-folds, as confirmed by Campbell Wheeler in his Ph.D.~Thesis~\cite{Wheeler:thesis,Wheeler:41}.

\section[How the q-series arise]{How the $\boldsymbol{q}$-series arise}
\label{sec.how}

\subsection{The quantum modularity conjecture}

In this section, we tell the rather amusing story of how we purely accidentally
found a $q$-series whose asymptotics near roots of unity agreed with the divergent
perturbative series arising from the volume conjecture and the quantum modularity
conjecture for the $4_1$ knot, and how a~series of further numerical experiments
led to the final picture that is described in this paper.

A knot $K$ has two famous quantum invariants, the (colored) Jones polynomial
$J_N^K(q) \in \BZ\big[q,q\i\big]$ and the Kashaev invariant $\langle K \rangle_N \in \Qbar$
for $N\in\BN$. (Both definitions will be
omitted since they aren't used here and can be found in many
places~\cite{Jones,K95,Tu1}.) Murakami--Murakami~\cite{MM} found that
$\langle K \rangle_N$ is the value of $J^K_N(q)$ at $q=\zeta_N$ and this
is the formula that we will need. For any knot it can in principle
be made explicit. For instance,
\begin{gather*}
\label{J041}
\la 4_1\ra_N  =  \sum_{n=0}^{N-1} \bigl|(\z_N ;\z_N)_n\bigr|^2
\end{gather*}
with $(q;q)_n:=\prod_{j=1}^n(1-q^j)$ being the usual $q$-Pochhammer symbol and
$\z_N={\rm e}^{2\pi {\rm i}/N}$. The Kashaev invariant can be extended equivariantly
to a function $\J$ on complex roots of unity. Moreover,
it is known by the work of Murakami and Murakami~\cite{MM} that the
(similarly defined) invariant $\J^K(-1/N)$
for any knot~$K$ is equal to the knot invariant $\la K \ra_N$ defined by
Kashaev~\cite{K95}. The famous volume conjecture of Kashaev
states that for any hyperbolic knot~$K$ the logarithm of
$\la K \ra_N$ is asymptotically equal to $CN$ as
$N$ tends to infinity, where~$C$ equals the (complexified) hyperbolic
volume of the knot divided by~$2\pi {\rm i}$. There are
very few cases for which the volume conjecture has been rigorously proved,
but for the~$4_1$ knot it is quite easy using
the Euler--Maclaurin formula and standard asymptotic techniques, because all
of the terms in~\eqref{J041} are positive,
and one finds the much more precise formula
\begin{gather}
\label{eq.K41}
\J^{4_1}\biggl(-\frac{1}{N}\biggr) \sim N^{3/2}
\Phih\biggl(\frac{2\pi {\rm i}}{N}\biggr)
\end{gather}
with $\Phih(h)$ defined by
\begin{gather*}
\Phih(h)= {\rm e}^{{\rm i}V/h} \Phi(h) ,
\end{gather*}
where $V$ is the hyperbolic volume of the knot
\begin{gather}
\label{eq.V41}
V =  \operatorname{Vol} \bigl(S^3\sm 4_1\bigr)
 =  2 \Im\big({\rm Li}_2\big ({\rm e}^{\pi {\rm i}/3}\big)\big)  =  2.0298829\dots,
\end{gather}
and where $\Phi(h)$ is the formal power series with algebraic coefficients
(which up to a common factor all lie in the trace field $\BQ\big(\sqrt{-3}\big)$
of the $4_1$ knot) having the form
\begin{gather}
\label{DefPhi}
\Phi(h) = \sum_{j=0}^\infty A_j h^j , \qquad
A_j =  \frac1{\sqrt[4]3} \biggl(\frac{1}{72\sqrt{-3}}\biggr)^j \frac{a_j}{j!}
\end{gather}
with $a_j\in\BQ$, the first values being given by
\begin{center}
\def\arraystretch{1.3}
\begin{tabular}{|c|c|c|c|c|c|c|c|c|}\hline
$j$ & $0$ & $1$ & $2$ & $3$ & $4$ & $5$ & $6$ & $7$
\\ \hline
$a_j$ & $1$ & $11$ & $697$ & $\frac{724351}{5}$
& $\frac{278392949}{5}$
& $\frac{244284791741}{7}$
& $\frac{1140363907117019}{35}$
& $\frac{212114205337147471}{5}$
\\ \hline
\end{tabular}
\end{center}
A proof of~\eqref{eq.K41} is given in~\cite{BD} and in~\cite{GZ:kashaev}.
A weaker asymptotic formula with $\Phi(h)$ replaced by its constant
term~$a_0$ was proved by Andersen and Hansen~\cite{AH}.

\subsection[A q-series G\_0(q)]{A $\boldsymbol{q}$-series $\boldsymbol{G_0(q)}$}

The surprising discovery that we made, completely by accident, is that there
is a close connection between the asymptotic expression occurring here and
the radial asymptotics of the function in the unit disk defined by
\begin{gather}
\label{Defg}
G_0(q)  =  (q;q)_\infty  \sum_{n=0}^\infty (-1)^n \frac{q^{n(3n+1)/2}}{(q;q)_n^3}
  =  1-q-2 q^2-2 q^3-2 q^4+q^6+\cdots .
\end{gather}
The infinite sum in~\eqref{Defg} occurred in the work of the first author
on the stability of the coefficients of the evaluation of the regular
quantum spin network~\cite[Section~7]{Ga:arbeit}, and in the course of a~numerical investigation of its asymptotics as~$q\to1$ we discovered
empirically the following:

\begin{observation}
\label{ob.1}
We have
\begin{gather}
\label{MainConj}
G_0\big ({\rm e}^{2 \pi {\rm i} \tau}\big) \sim  \sqrt \tau \bigl(\Phih(2\pi {\rm i}\tau)
- {\rm i} \Phih(-2\pi {\rm i}\tau)\bigr)
\end{gather}
to all orders in~$\tau$ as $\tau$ tends to~$0$ along any ray in the interior of the
upper half-plane.
\end{observation}

It was to achieve this simple statement that we included the
factor~$(q;q)_\infty$ in~\eqref{Defg}. The proofs of this observation and
the subsequent ones in this section are sketched in the appendix.
Our next discovery were two further formulas for~$G_0$ that we found empirically.

\begin{observation}
We have
\begin{gather}
\label{eq.3g}
G_0(q)  =  \frac1{(q;q)_\infty}\sum_{n,m=0}^\infty (-1)^{n+m}
\frac{q^{(n+m)(n+m+1)/2}}{(q;q)_n (q;q)_m}
 =  \sum_{n=0}^\infty (-1)^n \frac{q^{n(n+1)/2}}{(q;q)_n^{ 2}}  .
\end{gather}
\end{observation}

A proof of the above equation was given by S.~Zwegers (see
Appendix~\ref{sub.ZwegersIdentity}).
These expressions are of interest because, unlike the original series
in~\eqref{Defg} whose origin had no obvious connection with the $4_1$ knot,
these series \emph{are} related to it: the first one, which was shown to us by
T.~Dimofte, is typical of the series occurring in his work with Gaiotto and
Gukov \cite{DGG2,DGG1,Ga:index} on the 3D index of a triangulation, while
the second one is typical of those occurring in the work of Dimofte
and the first author on $q$-series associated to ideal triangulations of
cusped 3-manifolds~\cite{DG}.

Equation~\eqref{MainConj} turns out to be only a part of a bigger story. On
the one hand, the power series $\Phi(h)$ is only a special case at
$\a=0$ of the more general asymptotic series $\Phi_\a(h)$ ($\a\in\BQ$)
occurring in the modularity conjecture for $\J^{4_1}(q)$ made by the second
author in~\cite{Za:QMF} and play a~central role in our prior
paper~\cite{GZ:kashaev}. These asymptotic series appear in
the asymptotics of $G_0(q)$ for $q={\rm e}^{2\pi {\rm i}(\a+\tau)}$ as~$\tau\to 0$ in
a cone in the upper half-plane. This will be discussed in
Section~\ref{sec.rootsofunity} below. On the other hand, the $q$-series~$G_0(q)$
and the asymptotic formula~\eqref{MainConj} are related to the
Dimofte--Gaiotto--Gukov index and to the Hikami--Kashaev state integral.
We explain this next.

\subsection[The index, the state integral and a second q-series G\_1(q)]{The index, the state integral and a second $\boldsymbol{q}$-series $\boldsymbol{G_1(q)}$}

After describing the radial asymptotics of $G_0(q)$ at roots of unity,
our next step was to look for a connection between the power series $G_0(q)$
and the index of $4_1$. The index is an invariant of suitable ideal
triangulation introduced in \cite{DGG2, DGG1}. Necessary and sufficient
conditions for its convergence were established in \cite{Ga:index}
and its topological invariance was proven in \cite{GHRS}, leading in
particular to an invariant $\Ind_K(q)$ for any knot~$K$ \big(in equation (2)
of~\cite{GHRS}, this invariant was denoted by $I^{\mathrm{tot}}_K(q)$\big).
The index is defined as a sum over a lattice of products of the tetrahedron
index function
\[
\ID(m,e) =
\sum_{n=\max\{0,-e\}}^\infty (-1)^n \frac{q^{\h n(n+1) - (n+\h e )m}}{(q;q)_n(q;q)_{n+e}} .
\]
For the $4_1$ knot, the rotated index at $(0,0)$ (abbreviated simply by the
index below) is given by
\[
\Ind_{4_1}(q)  = \sum_{k_1,k_2 \in \BZ} \ID(k_1,k_2)\ID(k_2,k_1)
 =  1-8 q-9 q^2+18 q^3+46 q^4+90 q^5+\cdots .
\]

It seems quite natural to expect a relation between the $\Ind_{4_1}(q)$
and $G_0(q)$. This is encouraged by the illegitimate rewriting of $\Ind_{4_1}(q)$
as a 4-dimensional sum over the integers (which is divergent), but after some
rearrangement it decouples into the product of two two-dimensional sums
each of which is equal to $G_0(q)$. Nonetheless, when we performed experiments
no relation between the series $\Ind_{4_1}(q)$ and $G_0(q)^2$ was observed.

The key to finding the missing relation between $\Ind_{4_1}(q)$ and $G_0(q)$
turned out to involve the Andersen--Kashaev state integral associated to the $4_1$
knot~\cite{AK} and its factorization~\cite{GK:qseries} as a sum of products
of $q$-series and $\tq$-series.

State integrals appear in quantum hyperbolic geometry and in Chern--Simons theory
with complex gauge group pioneered by the work of Kashaev~\cite{AK,KLV},
Dimofte~\cite{Dimofte:complex, Dimofte:perturbative} and many other
researchers~\cite{DGLZ, Hi}. Their building block is the Faddeev quantum
dilogarithm, and a~suitable combinatorial ideal triangulation of a cusped hyperbolic
3-manifold~$M$ and the result is a~holomorphic function which is often
independent of the ideal triangulation, thus a topological invariant.
Below, we will use the state integral of the Andersen--Kashaev invariant
of a hyperbolic knot complement~\cite{AK}. In the normalization that we will
use this invariant is a holomorphic function $Z_M(\tau)$ on the cut plane
$\BC'$, and for the $4_1$ knot is given by (see~\cite[Section~11.4]{AK})
\begin{gather}
\label{Psi41}
Z_{4_1}(\tau)  =  \int_{\BR + {\rm i} \varepsilon} \Phi_{\sqrt{\tau}}(x)^{2}
 {\rm e}^{- \pi {\rm i} x^2} {\rm d}x, \qquad
\tau \in \BC' = \BC \sm (-\infty,0]
\end{gather}
\big(for convenience we write $Z_K$ in place of $Z_{S^3\sm K}$\big),
with small positive $\varepsilon$, where $\Phi_\mathsf{b}(x)$ is Faddeev's quantum
dilogarithm~\cite{Fa}
\begin{gather*}
\Phi_{\mathsf{b}}(x)  =  \frac{\bigl(-q^{1/2} {\rm e}^{2\pi \mathsf{b}x};q\bigr)_\infty}{
 \bigl(-\tq^{1/2} {\rm e}^{2\pi \mathsf{b}^{-1}x};q\bigr)_\infty},
\qquad q={\rm e}^{2 \pi {\rm i} \tau},\qquad \tq={\rm e}^{-2 \pi {\rm i}/\tau},\qquad \tau=\mathsf{b}^2 .
\end{gather*}
As is well known
(see for instance~\cite{Beem, GK:qseries}), the structure of the set of
poles of the quantum dilogarithm permits one to factorize this integral
as a finite sum of a product of functions of~$q$ and $\tq$ as above.
The answer here is given by the following theorem. Let $G_1(q)$ be the
$q$-hypergeometric series defined by
\begin{gather}
\label{eq.41deform}
\frac{(q {\rm e}^{\ve};q)_\infty^2}{(q;q)_\infty^2} 
\sum_{m=0}^\infty (-1)^m \frac{q^{m(m+1)/2} {\rm e}^{(m+1/2)\ve}}{(q {\rm e}^{\ve};q)_m^2}
 =  G_0(q) +\frac{\ve}{2} G_1(q) + O(\ve)^2
\end{gather}
and given explicitly by
\begin{align}
G_1(q) &  = \sum_{m=0}^\infty \frac{(-1)^m q^{m(m+1)/2}}{(q;q)_m^2}
 \biggl(\E_1(q)\+2\sum_{j=1}^m\frac{1+q^j}{1-q^j} \biggr) \nonumber
\\
&  =  1 - 7 q - 14 q^2 - 8 q^3 - 2 q^4 + 30 q^5 + 43 q^6 + 95 q^7 + 109 q^8
+\cdots, \label{eq.G}
\end{align}
where $\E_1(q)  ($``the non-modular Eisenstein series of weight~1''$)$
is the power series
\begin{gather}
\label{eq.ei1}
\E_1(q) = 1 - 4\sum_{n=1}^\infty\frac{q^n}{1-q^n}
 = 1 - 4\sum_{n=1}^\infty d(n) q^n,
\qquad d(n) = \text{\rm number of divisors of}~n.
\end{gather}

\begin{Theorem*}[\cite{GK:qseries}]
When $\operatorname{Im}(\tau) >0$, we have
\begin{equation}
\label{eq.prop1}
2{\rm i} (\tq/q)^{1/24} Z_{4_1}(\tau)  =  \tau^{1/2} G_1(q) G_0(\tq)
 - \tau^{-1/2} G_0(q) G_1(\tq) ,
\end{equation}
where $q={\rm e}^{2 \pi {\rm i} \tau}$ and $\tq={\rm e}^{-2 \pi {\rm i}/\tau}$.
\end{Theorem*}

The coefficients of~$G_0(q)$ and~$G_1(q)$ can be computed easily using that
\begin{gather}
\label{eq.G01coeff}
G_0(q) = \sum_{m=0}^\infty T_m(q), \qquad
G_1(q) = \sum_{m=0}^\infty R_m(q) T_m(q),
\end{gather}
where $T_m(q)$ and $R_m(q)$ are given by the recursion
\begin{gather*}
T_m(q) = -\frac{q^m}{(1-q^m)^2} T_{m-1}(q) , \qquad
R_m(q) = R_{m-1}(q)\+2 \frac{1+q^m}{1-q^m}
\end{gather*}
with initial conditions $T_0(1)=1$ and $R_0(q)=\calE_1(q)$. For instance,
we find
\begin{align*}
G_1(q)={}&
1-7 q-14 q^2-8 q^3-2 q^4+30 q^5+43 q^6+95 q^7+109 q^8+137 q^9+133 q^{10}
\\ &{}
+118 q^{11}+20 q^{12}-64 q^{13}-232 q^{14}-468 q^{15}-714 q^{16}-1010 q^{17}
-1324 q^{18} \\
&{}-1632 q^{19} -1878 q^{20} + \dots
-207821606967484464484714504354799 q^{1500} + \cdots.
\end{align*}
Quite by accident, when we compared the power series expansions of
$G_0(q)$, $G_1(q)$, and the index, we discovered the following.

\begin{observation}
\label{ob.3}
The three $q$-series $G_0(q)$, $G_1(q)$ and $\Ind_{4_1}(q)$ are related by
\begin{gather}
\label{eq.ind41}
\Ind_{4_1}(q)  =  G_0(q) G_1(q) .
\end{gather}
\end{observation}

A proof of equation~\eqref{eq.ind41} was communicated to us by T. Dimofte
and an additional proof follows from the results of
\cite[Section~5.3]{GGM:peacock}. This observation suggests that the $q$-series $G_0(q)$
and $G_1(q)$ are intimately related. Since we had already discovered a
relationship between the asymptotics of $G_0(q)$ as $q\to1$ and the power
series occurring in~\eqref{eq.K41} (Observation~\ref{ob.1}), it was natural
to make a similar numerical study of the asymptotics of $G_1(q)$ as~$q\to1$.
The result of this experiment, stated in the following observation, was
surprisingly simple.

\begin{observation}
\label{ob.4}
We have
\begin{gather}
\label{eq.Gas}
G_1\big ({\rm e}^{2 \pi {\rm i} \tau}\big)  \sim  \frac{1}{\sqrt{\tau}} \big(\Phih(2\pi {\rm i}\tau)
+ {\rm i} \Phih(-2\pi {\rm i}\tau)\big)
\end{gather}
to all orders in~$\tau$ as $\tau$ tends to~$0$ in a cone in the interior of the
upper half-plane.
\end{observation}

The right hand side of equation~\eqref{eq.41deform} defines
a sequence of power series (one for every power of $\ve$) the first two of
which are $G_0(q)$ and $G_1(q)/2$. This is analogous to the $\ve$-deformations
of linear differential equations studied for instance by Golyshev and
the second author~\cite{GZ:gamma,Za:hirzebruch}, and also analogous to the
theory of Jacobi forms, where $\ve$ plays the role of a Jacobi variable.
The connection between $\ve$-deformation and factorization of state integrals
is discussed further in Appendix~\ref{sub.estate} below.
One may wonder whether the $q$-series given by the coefficient of $\ve^2$
\big(or~$\ve^k$ for $k \geq 2$\big) in~\eqref{eq.41deform} has radial asymptotics given
by a variation of Observations~\ref{ob.1} and~\ref{ob.4}. A~relation was
recently found by Wheeler~\cite{Wheeler:thesis}.

We discovered empirically the following alternative $q$-series representation
for $G_1$, which is just a slight modification of the second formula for
$G_0$ given in~\eqref{eq.3g}.

\begin{observation}
For $|q|<1$ we have
\begin{gather}
\label{eq.Gsimple}
G_1(q) = \sum_{n=0}^\infty (-1)^n \frac{q^{\frac{n(n+1)}{2}} (6n+1)}{(q;q)_n^2}
 .
\end{gather}
\end{observation}

This was later proved in~\cite[Section~5.3]{GGM:peacock}.

\subsection[Holomorphic functions in C setminus R]{Holomorphic functions in $\boldsymbol{\BC\sm\BR}$}

The relation of the $q$-series $G_0$ and $G_1$ with the state integral given
in equation~\eqref{eq.prop1} brings out one more aspect to the $q$-series
$G_0$ and $G_1$, namely their extension outside the unit disk~${|q|>1}$.
This happens because on the one hand the state integral satisfies the symmetry
\begin{gather*}
Z_{4_1}(\tau) = Z_{4_1}\big(\tau^{-1}\big), \qquad \tau \in \BC\sm\BR
\end{gather*}
(which in turn follows from the corresponding symmetry of Faddeev's quantum
dilogarithm), and on the other hand the state integral is factorized in terms of
explicit $q$-hypergeometric series, which are guaranteed to be convergent
when $|q| \neq 1$. Indeed, the summand in last part of equation~\eqref{eq.3g}
is invariant under the replacement of $q$ by $q^{-1}$, and hence the formula of the
equation defines an extension of $G_0$ for $|q|>1$ which satisfies the
property $G_0(q)=G_0\big(q^{-1}\big)$. Likewise, equation~\eqref{eq.G}, together
with the convention that $E_1\big(q^{-1}\big)=-E_1(q)$ for $|q|>1$, defines an
extension of $G_1$ which satisfies the property $G_1(q) = -G_1\big(q^{-1}\big)$.
Summarizing, we have
\begin{gather*}
\label{eq.inout}
G_0(q) = G_0\big(q^{-1}\big), \qquad G_1(q) = -G_1\big(q^{-1}\big), \qquad q \in \BC, \qquad
|q| \neq 1
\end{gather*}
and equation~\eqref{eq.prop1} holds for $\tau \in \BC\sm\BR$.

\section[q-series and perturbative series]{$\boldsymbol{q}$-series and perturbative series}
\label{sec.qandh}

In this section, we discuss three further aspects of our pair $(G_0(q),G_1(q))$
of $q$-series. One is that their asymptotic expansions depend on a sector.
This seems to be a property of general $q$-hypergeometric series not observed before,
which is not only theoretically interesting, but also practically so, since
to numerically compute asymptotic expansions, we can choose rays with a~single
dominant asymptotics, making the numerical computation much easier. From that point
of view, the numerical asymptotics when $q \in [0,1)$ tends to $1$ is a very resonant
situation.

A second aspect is that bilinear combinations of factorially divergent series give
convergent power series. These bilinear combinations are motivated by the factorization
of state-integrals, combined by the asymptotic expansions of our $q$-series, and
lead to explicit formulas for the Taylor series expansions of state-integrals at
rational numbers, which subsequently have been proven in~\cite{GK:evaluation}.

The third aspect is that the asymptotic analysis of our $q$-series can be extended
to any complex root of unity. This is hardly a surprise, and relates the asymptotic
expansions of the pair $(G_0(q),G_1(q))$ as $q$ approaches a root of unity to the
asymptotic expansions of the Kashaev invariant in the quantum modularity
conjecture of the second author~\cite{Za:QMF}.

\subsection{Asymptotics of holomorphic functions in sectors}

Since we will be considering functions of $q$ on $|q| \neq 1$ as well as
functions of $\tau \in \BC\sm\BR$ with $q={\rm e}^{2\pi {\rm i} \tau}$,
we will use capital letters for functions $F(q)$ of $q$ with $|q| \neq 1$
and small letters for the corresponding functions $f(\tau):=F\big ({\rm e}^{2\pi {\rm i} \tau}\big)$ of
$\tau \in \BC\sm\BR$. For instance, we have
\begin{gather*}
g_0(\tau) = G_0\big ({\rm e}^{2 \pi {\rm i} \tau}\big), \qquad g_1(\tau) = G_1\big ({\rm e}^{2 \pi {\rm i} \tau}\big),
\qquad \tau \in \BC\sm\BR
\end{gather*}
and Observations~\ref{ob.1} and~\ref{ob.4} can be written in the form
\begin{gather}
\label{ggas}
g_0(\tau)  \sim  \sqrt \tau \big(\Phih(2 \pi {\rm i} \tau)
- {\rm i} \Phih(-2 \pi {\rm i} \tau)\big), \qquad
g_1(\tau)  \sim  \frac{1}{\sqrt{\tau}} \big(\Phih(2 \pi {\rm i} \tau)
+ {\rm i} \Phih(-2 \pi {\rm i} \tau)\big)
\end{gather}
as $\tau \in\BC\sm\BR$ goes to 0 in a cone in the interior of the
upper half-plane.
We emphasize here that we are not only considering limits as $q\to1$
radially, which would correspond to taking $\tau={\rm i}\epsilon$ with a positive
real number $\epsilon$ tending to zero, but are also allowing~$\tau$ to tend
to~0 at an fixed angle. This is important when actually doing the numerical
experiments since often (and also here) the limit when one moves along the
imaginary axis only is hard to recognize because the two terms
in~\eqref{eq.Gas} are both oscillatory and have the same order of magnitude,
so that they interfere with one another, and it is only possible to see the
numerical structure clearly when one allows oneself more freedom.
The two asymptotic series $\Phih(2 \pi {\rm i} \tau)$ and ${\rm i} \Phih(-2 \pi {\rm i} \tau)$
partition the upper half plane into two sectors $S_1\colon \arg(\tau) \in (0,\pi/2]$ and
$S_2\colon\arg(\tau) \in [\pi/2,\pi)$; see Figure~\ref{f.41rates}. In the interior of~$S_1$,
$\Phih(2\pi {\rm i}\tau)$ dominates exponentially, and the reverse happens in $S_2$,
while on the common ray $\arg(\tau)=\pi/2$ both functions have oscillatory growth.

\begin{figure}[!htpb]\centering
\includegraphics[height=0.17\textheight]{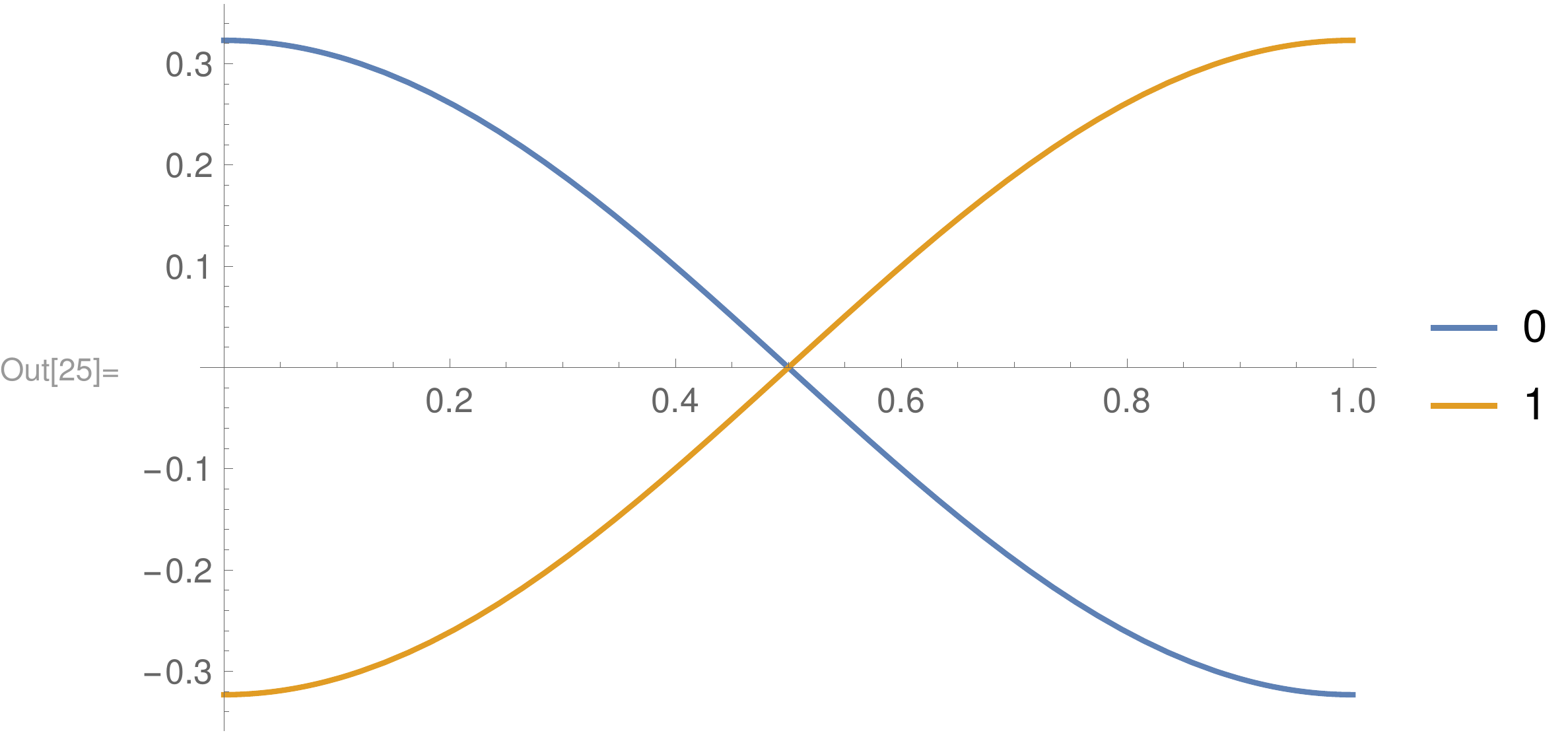}
\caption{A plot of the growth rates $\Re({\rm i}V/(\pm2\pi {\rm i}\t))$ of $\Phih(\pm2\pi {\rm i}\tau)$
 for $ \arg(\t)=\pi \theta$ with $0<\theta<1$ and~$|\t|$ fixed. $0$ means $+$ branch
 and $1$ means $-$ branch.
The branches cross at $0.5$ and partition the interval $(0,1)$ in two sectors.}\label{f.41rates}
\end{figure}

On a fixed ray, the asymptotic statements of equation~\eqref{ggas} involves
combinations of series with different growth rates, and it would appear at first
sight that the coefficient in front of the dominated series in~\eqref{ggas} is
meaningless. However, the refined optimal truncation of~\cite{GZ:kashaev, GZ:optimal}
allows us to make numerical sense of the both divergent series $\Phi(\pm2\pi {\rm i}\t)$
with a relative error that is exponentially rather than merely polynomially
small compared to the leading term, and then we can ``see'' both terms in~\eqref{ggas}.

We can also try to take a linear combination of the two equations in~\eqref{ggas}
to get new holomorphic functions $w^{(\s_1)}(\t)$ and~$w^{(\s_2)}(\t)$ whose
asymptotic behavior near the origin gives each of the individual completed series
$\Phih(\pm2\pi {\rm i}\t)$ separately. Specifically, if we define a holomorphic
vector-valued function
\smash{$w(\t)=\begin{psmall} w^{(\s_1)}(\t) \\ w^{(\s_2)}(\t) \end{psmall}$} by
\begin{gather}
\label{eq.41g2w}
w(\t)  =  \frac12 \begin{pmatrix} 1 & \hphantom{-}1 \\ 1 & - 1 \end{pmatrix}
\begin{pmatrix} \tau^{-1/2} g_0(\t) \\ \tau^{1/2} g_1(\t) \end{pmatrix} ,
\qquad
\begin{pmatrix} \tau^{-1/2} g_0(\t) \\ \tau^{1/2} g_1(\t) \end{pmatrix}
  =  \begin{pmatrix} 1 & \hphantom{-}1 \\ 1 & -1 \end{pmatrix} w(\t) ,
\end{gather}
then equation~\eqref{ggas} might seem to imply the asymptotic statements
\begin{gather}
\label{NoSectors}
w^{(\s_1)}(\t)\sim\Phih(2\pi {\rm i}\tau),\qquad
w^{(\s_2)}(\t)\sim-{\rm i} \Phih(-2\pi {\rm i}\tau)
\end{gather}
to all orders in both quarter-planes~$S_1$ and~$S_2$.
In any case, the passage from $g$ to $w$ has several other nice consequences.
The first is a very simple formula for
the index, namely
\begin{gather*}
\Ind_{4_1}\big ({\rm e}^{2\pi {\rm i}\tau}\big) = w^{(\s_1)}(\tau)^2 - w^{(\s_2)}(\tau)^2
\end{gather*}
(combine equations~\eqref{eq.ind41} and~\eqref{eq.41g2w}), which when combined
with Observation~\ref{ob.6} gives the asymptotics of the 3D index
when $\tau$ tends to zero on the vertical axis.
The second, obtained by combining equation~\eqref{eq.prop1}
with Observation~\ref{ob.3}
as $\tau\to 0$, and using the fact that \hbox{$g_0(\ti\tau) \sim G_0(0)=1$} and
\hbox{$g_1(\ti\tau) \sim G_1(0)=1$}, is the asymptotic formula
\begin{gather*}
-(\tq/q)^{1/24} Z_{4_1}(\tau) \sim  \Phih(-2\pi {\rm i}\tau),
\qquad\tau\to0^+ .
\end{gather*}
In other words, the state integral as $\tau\to 0$ exponentially decays with the
fastest possible rate and with an asymptotic expansion matching to all orders that
of the Kashaev invariant at~${q=1}$. This is a version of the Volume Conjecture for
the state integral which has recently been established for knot complements with
suitable ideal triangulations in~\cite{AGK}.

However, equation~\eqref{NoSectors} is not quite true.
Instead, we find that it is true in a wide neighborhood of the imaginary
axis, but fails when $\t$ approaches~0 from very near the positive or negative real
axis. More precisely, what we find numerically is the following

\begin{observation}
\label{ob.6}
The first asymptotic equation in~\eqref{NoSectors} holds to all orders in~$\t$ as $\t$
tends to~$0$ along a ray with argument between~$0$ and~$\pi-0.11$, but fails when the
argument is larger, while the second equation holds to all orders if~$\t$ tends to~$0$
along a ray with argument between $0.11$ and~$\pi$, but fails for small arguments.
\end{observation}
As an illustration, for $\t=\frac{-10+{\rm i}}{100000}$ we find
\begin{alignat*}{3}
&w^{(\s_1)}(\t) \approx (-3.656 -4.937\thin {\rm i})\times10^{-1313},\qquad
&&\Phih(2\pi i\tau) \approx (4.351+2.821\thin {\rm i})\times 10^{-1390},&\\
& w^{(\s_2)}(\t) \approx (-6.057-9.343\thin {\rm i})\times10^{1388},\qquad&&
-{\rm i} \Phih(-2\pi {\rm i}\tau) \approx (-6.057-9.343\thin {\rm i})\times10^{1388},&
\end{alignat*}
so that $w^{(\s_2)}(\t)$ is indeed asymptotically close to~$-{\rm i}\Phih(-2\pi {\rm i}\tau)$
(and in fact their ratio equals~1 numerically to over 200 digits), but
$w^{(\s_1)}(\t)$ is not at all close to~$\Phih(2\pi {\rm i}\tau)$. On the other hand,
the ratio of $w^{(\s_1)}(\t)$ to $w^{(\s_2)}(\t)$ is extremely close to~$3\thin\tq$,
where $\tq:={\rm e}^{-2\pi {\rm i}/\t}$, and the corrected value
$w^{(\s_1)}(\t)-3\tq w^{(\s_2)}(\t)$ now coincides with~$\Phih(2\pi {\rm i}\tau)$ with
a relative accuracy of more than 200 digits. In other words, in this region
$w^{(\s_1)}(\t)$ is always asymptotically very close to
$\Phih(2\pi {\rm i}\tau)+3\thin {\rm i}\thin\tq\thin\Phih(-2\pi {\rm i}\tau)$, but there is a phase
transition on the line $ \arg(\t)=\arctan\big(V/2\pi^2\big)=0.10247\dots$, where the two
terms in this new approximation have the same order of magnitude as~$\t\to0$.
If we continue further to the left, then there is a new phase transition at
$\arg(\t)=\arctan\big(V/4\pi^2\big)=0.05137\dots$, where we need a further correction
term~$18\thin\tq^2\thin w^{(\s_2)}(\t)$ and similarly if we go further we find phase
transitions whenever $\arg(\t)=\arctan\big(V/2\pi^2m\big)$, where $\tq^m\Phih(-2\pi {\rm i}\tau)$
and $\Phih(-2\pi {\rm i}\tau)$ are of the same order of magnitude, the correction needed
at $\t=\frac{-40+{\rm i}}{100000}$ for instance being
$\big(3\tq+18\tq^2+99\tq^3+555\tq^4\big)w^{(\s_2)}(\t)$, which makes $w^{(\s_1)}$ agree with~$\Phih(2\pi {\rm i}\t)$
with a relative error of~$10^{-148}$ as opposed to the huge
$10^{+531}$ that we obtain without any correction. Note that we cannot find these
higher-order corrections in~$\tq$ by looking for a $\tq$-power series linear
combination of $\Phih(-2\pi {\rm i}\tau)$ and $\Phih(-2\pi {\rm i}\tau)$ that is very close
to $w^{(\s_1)}(\t)$, because even with improved optimal truncation we cannot evaluate
$\tq^m\Phih(-2\pi {\rm i}\tau)$ to the required degree of precision, but
since~$w^{(\s_1)}(\t)$ and~$w^{(\s_2)}(\t)$ are given in terms of convergent power
series that can be computed to any desired precision, we can find successive
terms of a power series $a=a_\pm(q)$ making $w^{(\s_1)}-aw^{(\s_2)}$ agree with
$\Phih(-2\pi {\rm i}\tau)$ to all orders $\t$ approaches the real line with any
argument between~0 and~$\pi$, and similarly (by studying the power series
near the positive real axis) another $\Z[[\tq]]$-power series linear combination
of $w^{(\s_1)}$ and $w^{(\s_2)}$ that agrees to all orders with $\Phih(2\pi {\rm i}\tau)$
in the entire upper half-plane. Both linear combinations are determined by these
requirements only up to multiplication of the whole expression by a power series
in~$\tq$ starting with~1. We will see later in Section~\ref{sec.matrix} why this
happens and how to find canonical $\Z[[\tq]]$-linear combinations of
$\t^{-1/2}g_0$ and~$\t^{1/2}g_1$ -- see equation~\eqref{W410-desc} below.

\subsection{From divergent to convergent power series}
\label{sub.convergent}

The third interesting corollary of Observations~\ref{ob.1} and~\ref{ob.4}
is obtained by combining them with equation~\eqref{eq.prop1} and the
fact that $Z_{4_1}(\t)$ is holomorphic in the cut plane~$\BC'$, since this
leads to startling predictions regarding the factorially divergent formal
power series $\Phi(h)\in\BR[[h]]$. Specifically, using the factorization of the state
integral given in~\eqref{eq.prop1}, the fact that each~$w^{(\s_j)}(\t)$ is a linear
combination of the functions $\tau^{-1/2}g_0(\tau)$ and $\tau^{1/2}g_1$, and the fact
that $g_0$ and $g_1$ are 1-periodic, we can re-express the state integral in terms
of $w$ as follows:
\begin{gather}
\label{eq.41swW}
(\tq/q)^{1/24} Z_{4_1}(\tau)  =
- w^{(\s_1)}(\tau-1)w^{(\s_2)}\biggl(\frac{\t-1}\t\biggr)
\+ w^{(\s_2)}(\tau-1)w^{(\s_1)}\biggl(\frac{\t-1}\t\biggr) .
\end{gather}
The fact that the state integral is holomorphic in~$\BC'$ implies that the right-hand
side of~\eqref{eq.41swW} has a Taylor expansion around $\t=1$ with radius of
convergence~1. However, this is wasteful because it uses only the holomorphy of
$Z_{4_1}$ in the disk~$|\t-1|<1$. If we use its holomorphy, first in
$\{\Re(\t)>0\}$ and then in all of~$\BC'$, then by making the changes of variables
\begin{gather}
\label{eq.taux0}
\tau  =  1+u  = \frac{1+v}{1-v}  = \biggl(\frac{1+w}{1-w}\biggr)^2 ,
\end{gather}
which give biholomorphic maps between the unit $u$-, $v$- and $w$-disks and the
sets $\{|\t-1|<1\}$, $\{\Re(\t)>0\}$ and $\BC'$, respectively, we find:

\begin{Corollary} Let $C  =  V/2\pi  =  0.3230659\dots$ and
$\Phi(x)\in\BR[[x]]$ be given by~\eqref{DefPhi}.
Each of the three formal power series $Q(u)\in\BR[[u]]$, $R(v)\in\BR\big[\big[x^2\big]\big]$
and $S(w)\in\BR\big[\big[w^2\big]\big]$ defined by
\begin{subequations}\label{eq3.6}
\begin{gather}
\label{DefQ}
Q(u)  =
{\rm e}^{-C} \Phi(2 \pi {\rm i} u)\Phi\biggl(-\frac{2 \pi {\rm i} u}{1+u}\biggr)
 - {\rm e}^{C} \Phi\biggl(\frac{2 \pi {\rm i} u}{1+u}\biggr)\Phi(-2 \pi {\rm i} u) ,
\\
\label{DefR}
R(v)  =
{\rm e}^{-C} \Phi\biggl(\frac{4 \pi {\rm i} v}{1-v}\biggr)
\Phi\biggl(-\frac{4 \pi {\rm i} v}{1+v}\biggr)
 - {\rm e}^{C} \Phi\biggl(-\frac{4 \pi {\rm i} v}{1-v}\biggr)
\Phi\biggl(\frac{4 \pi {\rm i} v}{1+v}\biggr) ,
\\
\label{DefS}
S(w)  =
{\rm e}^{-C} \Phi\biggl(\frac{8 \pi {\rm i} w}{(1-w)^2}\biggr)
\Phi\biggl(-\frac{8 \pi {\rm i} w}{(1+w)^2}\biggr)
 -
{\rm e}^{C} \Phi\biggl(-\frac{8 \pi {\rm i} w}{(1-w)^2}\biggr)
\Phi\biggl(\frac{8 \pi {\rm i} w}{(1+w)^2}\biggr)
\end{gather}
\end{subequations}
has radius of convergence $1$.
\end{Corollary}

Note that the original formulas obtained from~\eqref{eq.41swW} would have had
$\Phih$'s instead of~$\Phi$'s and would not have had the scalar factors ${\rm e}^{\pm C}$,
which arise from a cancellation of an exponentially large and an exponentially
small prefactor. This also means that each of the three power series~$Q$,~$R$ and~$S$
has coefficients in the ring~$\Q\big(\pi,\sqrt{-3},{\rm e}^C\big)$.

What the corollary says is that, although the original power series $\Phi(x)$
occurring in the asymptotic expansion of the Kashaev invariant $\la4_1\ra_N$ was
factorially divergent, each of the combinations~$Q$, $R$ and $S$ defined
by~\eqref{eq3.6} are convergent power series with radius of convergence~1.
This can be seen dramatically in following table showing the growth of the
coefficients (rounded), part of which was already given in~\cite[equations~(35) and~(83)]{GZ:kashaev}:
$$\def\arraystretch{1.1}
\begin{array}{|c|cccc|} \hline
k & 0 & 50 & 100 & 150 \\ \hline
\;[h^k]\Phi(h)\; & 0.75 & 6.7 \cdot 10^{71}\vphantom{A^{B^C}} & 3.1 \cdot 10^{174}
& 7.4 \cdot 10^{283} \\
\;[v^k]Q(v)\; & -0.379 & 0.012 & -0.007 & 0.002 \\
\;[u^k]R(u\;) & -0.380 & -0.037 & 0.009 & -0.001 \\
\;[w^k]S(w)\; & -0.379 & -52068.5 & -43932564.0 & -75312313899.2 \\ \hline
\end{array}
$$
Note that the fact that the coefficients of $S$, although very much smaller than
those of~$\Phi$, are much larger than those of $Q$~and~$R$, does not mean that $S$
is the worst of these three series, but actually the best one, since the larger growth
reflects the fact that the unit $w$-disk corresponds to the entire domain of
holomorphy~$\BC'$ of the state integral rather than a subset like the two other series,
and that consequently this power series has essential singularities on the entire
unit circle rather than at only one or two points. (This observation was already
made in~\cite{GZ:kashaev}.)

\subsection[The asymptotics of G\_0(q) and G\_1(q) at roots of unity]{The asymptotics of~$\boldsymbol{G_0(q)}$ and~$\boldsymbol{G_1(q)}$ at roots of unity}
\label{sec.rootsofunity}

Observations~\ref{ob.1} and~\ref{ob.4} express the asymptotics of
the functions $G_0(q)$ and $G_1(q)$ at $q=1$ in terms of the
series $\Phih^{4_1}(h)$ which appears in the asymptotics of the Kashaev
invariant at $q=1$. We now extend the above observation to all roots
of unity using the series $\Phi^{4_1}_{\a}(h)$ that appear in the quantum
modularity theorem of the Kashaev invariant of the $4_1$
knot~\cite{GZ:kashaev}. Let us briefly recall the latter. Let
\[
\J \colon \ \BQ \to \overline{\BQ} \subset \BC
\]
denote the extension of the Kashaev invariant of $4_1$ \cite{K95} to $\BQ$
where $\J(-1/N)=\la K \ra_N$. The quantum modularity theorem for the $4_1$
knot asserts that for every matrix $\gamma=\sma abcd\in\SL_2(\Z)$ we have
\begin{gather}
\label{eq.QMCg}
\J\biggl(\frac{aX+b}{cX+d}\biggr) \sim (cX+d)^{3/2}
\Phih_{a/c}\biggl(\frac{2\pi {\rm i}}{c(cX+d)}\biggr)\J(X)
\end{gather}
to all orders in $1/X$ as $X\to\infty$ in~$\Q$ with bounded denominator
where $\a=a/c$,
\begin{gather*}
\Phih_\a(h)  =  {\rm e}^{{\rm i}V/c^2h} \Phi_\a(h)
\end{gather*}
and $\Phi_\a(h)$ is a power series with algebraic coefficients.
Various refinements of the quantum modularity conjecture were discussed in
detail in~\cite{GZ:kashaev}. Since $\J$ is 1-periodic (i.e., defined for
$\a \in \BQ/\BZ$), it follows that the series $\Phih_\a(h)$ depends on
$\a \in \BQ/\BZ$.

The reflection of the quantum modularity statement~\eqref{eq.QMCg} for the power
series $g_0$ and~$g_1$ is the following extension of equation~\eqref{ggas},
in which we have set $\t=\a+\ve/c $:

\begin{observation}
 For a rational number $\a=a/c$, we have
\begin{subequations}\label{eq.41}
 \begin{align}
 \label{eq.41cuspsg}
 g_0(\a+\ve/c) &  \sim  \sqrt \ve \bigl(\Phih_{-\a}(2 \pi {\rm i} \ve)
  -  {\rm i} \Phih_\a(-2 \pi {\rm i} \ve)\bigr) , \\
\label{eq.41cuspsG}
g_1(\a+\ve/c) &  \sim  \frac{1}{\sqrt{\ve}} \bigl(\Phih_{-\a}(2 \pi {\rm i} \ve)
+ {\rm i} \Phih_\a(-2 \pi {\rm i} \ve)\bigr)
\end{align}
\end{subequations}
to all orders in~$\ve$ as $\ve \in \BC\sm\BR$ tends to~$0$ in a cone in the interior
of the upper half-plane.
\end{observation}

Finally, we reformulate the asymptotic expansions given in
equations~\eqref{eq.41} in a~way that
resembles the quantum modularity conjecture. Consider the vector-valued holomorphic
function $g=\bigl(\smallmatrix g_0 \\ g_1 \endsmallmatrix\bigr)$ on $\BC\sm\BR$,
where $g_0$ and $g_1$ are declared to have weights $-1/2$ and~$1/2$, and define
the corresponding vector-valued ``slash operator'' by
\[
\bigl(g\bigr|\g)(\t)  =  \begin{pmatrix} (c\tau+d)^{1/2} g_0(\g\t) \\
 (c\tau+d)^{-1/2} g_1(\g\t) \end{pmatrix}
 \]
for $\g=\sma abcd\in\SL_2(\R)$, where $\g\t=\frac{a\t+b}{c\t+d}$ as usual.
Then equations~\eqref{eq.41cuspsg} and~\eqref{eq.41cuspsG} can be written in the
equivalent form

\begin{observation}
 \label{ob.8}
 For any $\g=\sma abcd$ in $\SL_2(Z)$, we have
\begin{gather*}
\bigl(g|\gamma\bigr)(\t) \sim \begin{pmatrix} 1 & -1 \\ 1 & \hphantom{-}1 \end{pmatrix}
\Phih^\cdot_\a\biggl(\frac{2\pi {\rm i}}{c(c\t+d)}\biggr),
\qquad \t\in \BC\sm\BR, \qquad |\Im(\t)| \to \infty
\end{gather*}
to all orders in $1/\t$, where
$\Phih^\cdot_\a(h)=\Bigl(\smallmatrix \Phih_\a(h) \\ {\rm i} \Phih_\a(-h)
\endsmallmatrix\Bigr)$.
\end{observation}

Notice that Observation~\ref{ob.8} has a corollary generalizing
the one given in Section~\ref{sub.convergent}, giving linear combinations of two
products of a $\Phi_a$-series and a $\Phi_{-1/\a}$-series with radius of
convergence~1 for any rational number~$\a$, and not just for $\a=1$ as before.
We leave the details to the reader.

\subsection{The quadratic relation}

We now describe some new phenomena that we observed using other knots.
The knot $4_1$ was amphicheiral and hence special: in general one should expect
an $r$-tuple of pairs of $q$-series, one on each half-plane, hence a total
of $2r$ $q$-series. (We will see in a later Section \ref{sub.matrix} the topological
meaning of this number~$r$). However, in the case of the $4_1$ knot, the four
$q$-series are actually two, each appearing twice, due to the amphicheirality of the
$4_1$ knot. On the other hand, the factorization integral for the $5_2$ knot
and for the $(-2,3,7)$ pretzel knot gives a total of 6 and 12 $q$-series.
For each knot, the collection of these $q$-series satisfies one quadratic relation,
which is trivial for the case of the $4_1$ knot.

Let us illustrate the quadratic relation using the $5_2$ knot as an example.
The Andersen--Kashaev state integral of the $5_2$ knot is given
by~\cite[equation~(39)]{AK}
\begin{gather*}
Z_{5_2}(\tau)  =  \int_{\BR + {\rm i} \varepsilon} \Phi_{\sqrt{\tau}}(x)^3
 {\rm e}^{- 2\pi {\rm i} x^2} {\rm d}x, \qquad \tau \in \BC' .
\end{gather*}
In~\cite{GK:qseries}, by the same type of
residue calculation as in the $4_1$ case, it is shown that $Z_{5_2}$ has
the decomposition
\begin{gather}
\label{eq.52}
2{\rm e}^{3{\rm i}\pi/4}(\tq/q)^{1/8} Z_{5_2}(\tau)  =
\tau h_2(\tau) h_0\big(\tau^{-1}\big) + 2h_1(\tau)h_1\big(\tau^{-1}\big)
\+ \frac1\tau h_0(\tau) h_2\big(\tau^{-1}\big)
\end{gather}
for $\tau \in \BC\sm\BR$, where
\begin{gather}
\label{eq.hqqt}
h_j(\t)  =  (\pm1)^j H_j^\pm\bigl ({\rm e}^{\pm2\pi {\rm i}\t}\bigr)\qquad
\text{for} \quad \pm\Im(\t)>0
\end{gather}
are holomorphic functions in a half-plane and $H^\pm_j(q)\in\Z[[q]]$ are
$q$-series with coefficients in~$1/6 \BZ$ defined by
\begin{gather}
\frac{(q {\rm e}^{\ve};q)_\infty^3}{(q;q)_\infty^3}
 \sum_{m=0}^\infty \frac{q^{m(m+1)} {\rm e}^{(2m+1)\ve}}{(q {\rm e}^{\ve};q)_m^3}
\nonumber \\
 \qquad {}=  H^+_0(q) +\ve H^+_1(q)
+ \frac{\ve^2}{2} \left(H^+_2(q)+\frac{1}{6} \calE_2(q) H^+_0(q)\right) + \text O(\ve)^3, \nonumber
\\
\frac{(q;q)_\infty^3}{(q {\rm e}^{-\ve};q)_\infty^3}
 \sum_{n=0}^\infty (-1)^n \frac{q^{n(n+1)/2} {\rm e}^{(n+1/2)\ve}}{(q {\rm e}^{\ve};q)_n^3}\nonumber\\
 \qquad{} =  H^-_0(q) \+ \ve H^-_1(q)
+ \frac{\ve^2}{2} \left(H^-_2(q)
+\left(\frac{1}{4}-\frac{1}{6} \calE_2(q)\right) H^-_0(q)\right) + O(\ve)^3, \label{eq.52deform}
\end{gather}
whose first few terms are given by
\begin{gather*}
 H^+_0(q)   =
 1 + q^2 + 3q^3 + 6q^4 + 10q^5 + \cdots,
\\
 H^-_0(q)   =  1 - q - 3q^2 - 5q^3 - 7q^4 - 6q^5 + \cdots,
\\
H^+_1(q)   =
 1 - 3q - 3q^2 + 3q^3 + 6q^4 + 12q^5 + \cdots,
\\
H^-_1(q)   =
\tfrac12 \bigl(
1 - 9q -21q^2 -19q^3 - 9q^4 + 54q^5\bigr) + \cdots,
\\
H^+_2(q)   =
\frac{5}{6} - 5q + \frac{53}{6} q^2 + \frac{117}{2} q^3 + 117q^4
+ \frac{601}{3} q^5 + \cdots,
\\
H^-_2(q)   =
\frac{1}{6} - \frac{37}{6} q - \frac{17}{2} q^2 + \frac{115}{6}q^3
+ \frac{389}{6}q^4 + \cdots,
\end{gather*}
and whose further properties are given in Appendix~\ref{sub.q52}.
Here, $\calE_2(q) = 1-24 \sum_{n \geq 1} \frac{q^n}{(1-q^n)^2}$ is the weight~2
Eisenstein series and $(x;q)_\infty=\prod_{k=0}^\infty \big(1-q^k x\big)$.

The index of the $5_2$ knot is given by the following expression:
\begin{align*}
\Ind_{5_2}(q) &{}= \sum_{k_1, k_2, k_3 \in \BZ}
\ID(-k_1, k_1 - k_2) \ID(-k_1, k_1 - k_2 - k_3)
\ID(2 k_1 - 2 k_2 -k_3, -k_1) \\
&{} =  1 - 12 q + 3 q^2 + 74 q^3 + 90 q^4 + 33 q^5 - 288 q^6 - 684 q^7
- \cdots .
\end{align*}
The next observation (a proof follows from results
of~\cite[Section 5.3]{GGM:peacock}) was expected given what we knew from the case
of the $4_1$ knot.

\begin{observation}
The $q$-series $H_j$ are related to the index by
\begin{gather*}
\Ind_{5_2}(q)  =  2 H^+_1(q) H^-_1(q) .
\end{gather*}
\end{observation}

The next observation, a quadratic relation among the 3 pairs of $q$-series
was unexpected and found by accident. This relation could
not be seen in the case of the $4_1$ knot, since it reduces to the empty
equation $G_0(q)G_1(q)-G_1(q)G_0(q)=0$, as a consequence of the fact that the
$4_1$ knot is amphicheiral.

\begin{observation}
The $q$-series $H_j$ satisfy the quadratic relation
\begin{gather*}
H^+_0(q)H^-_2(q) - 2H^+_1(q)H^-_1(q) \+ H^+_2(q)H^-_0(q)  =  0 .
\end{gather*}
\end{observation}

We now discuss the asymptotics of the six $q$-series of the $5_2$
knot. Just as in the case of the $4_1$ knot, the asymptotics of
$h_j(\tau)$ as $\tau\in \BC\sm\BR$ tends to zero in a ray
are given by a rational linear combination of three asymptotic
series $\Phih^{(\s)}(h)$ that appear in the quantum modularity conjecture
of the $5_2$ knot~\cite{GZ:kashaev}, where $\s$ denotes one of the
three embeddings of the trace field of the $5_2$ knot (the cubic field of
discriminant $-23$ generated by $\xi$ with $\xi^3-\xi^2+1=0$). Each embedding
corresponds to a boundary parabolic $\SL_2(\BC)$ representations
of the fundamental group of the complement of the knot, with the convention
that $\s_1$, $\s_2$ and $\s_3$ denotes the geometric embedding,
(corresponding to $\Im(\xi)<0$, its complex conjugate, and the real
embedding of the trace field).
When $\tau$ approaches zero in a fixed generic ray, the three asymptotic
series $\Phih^{(\s_j)}(h)$ have different growth rates and this divides
each of the upper and lower half-plane into four sectors shown in
Figure~\ref{f.52rates}.

\begin{figure}[t]\centering
\includegraphics[height=0.17\textheight]{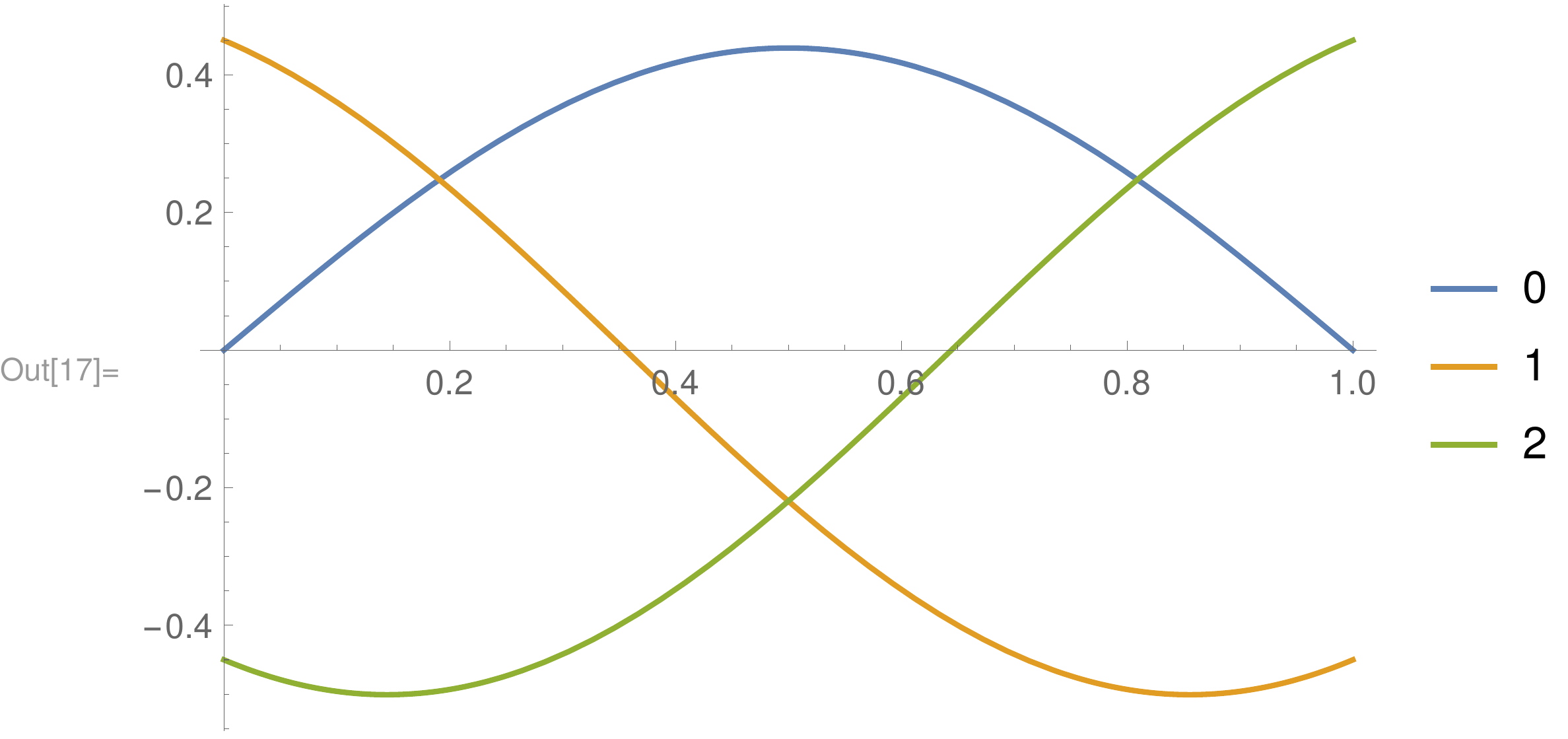}
\caption{A plot of the growth rates $\Re(\VC(\rho_j)/2\pi {\rm i}\tau)$
 of $w^{(\s_j)}(x)$ defined in equation~\eqref{eq.52was} for~${j=1,2,3}$ where
 $\arg(\tau)=\pi\th$ and $\theta \in (0,\pi)$.
 The branches cross at $0.19$, $0.5$, $0.81$ and partition the interval
$[0,1]$ in four sectors.}
\label{f.52rates}
\end{figure}

Just as in the case of the $4_1$ knot, the refined optimal truncation
of~\cite{GZ:kashaev} finds in each sector $R$ a unique matrix $M_R$ such that
$h(\tau) \sim M_R \Phih(2\pi {\rm i} \tau)$ as $\tau \in R$ and
$\tau \to 0$, where
\[
h=\begin{pmatrix}  \tau^{-1} h_0\\ h_1\\ \tau h_2 \end{pmatrix} \qquad
\text{and}\qquad
\Phih=\begin{pmatrix} \Phih^{(\s_1)} \\ \Phih^{(\s_3)} \\
\Phih^{(\s_2)} \end{pmatrix}.
\] Using $108$ exact coefficients of the
power series $\Phih$ and refined optimal truncation, we found the following.
\begin{observation}
We have
\begin{gather}
\label{eq.hv4}
h(\tau) \sim \begin{cases}
N_+ \Phih(2\pi {\rm i}\tau) & \text{when} \ \arg(\tau) \in (0,0.19) ,\\
N_- \Phih(2\pi {\rm i}\tau) & \text{when} \ \arg(\tau) \in (-\pi/2,0),
\end{cases}
\end{gather}
where
\begin{gather*}
N_+=
\begin{pmatrix}
\hphantom{-}1/2 & 1/2 & \hphantom{-}1 \\
 \hphantom{-}0 & 1/2 & \hphantom{-}1/2 \\
 -1/12 & 5/12 & -2/3
 \end{pmatrix}, \qquad
N_- =
\begin{pmatrix}
-1/2 & -1/2 & \hphantom{-}1/2 \\
\hphantom{-}3/4 & -1/4 & -1/4 \\
-13/12 & -1/12 & \hphantom{-}1/12
\end{pmatrix} .
\end{gather*}
\end{observation}

Inverting the matrices $N_\pm$ we obtain a vector
\smash{$w = \biggl(\smallmatrix w^{(\s_1)} \\ w^{(\s_3)} \\ w^{(\s_2)}
\endsmallmatrix\biggr)$} of holomorphic functions on $\BC\sm\BR$
\begin{gather*}
w(\tau) = \begin{cases}
N_+^{-1} h(\tau) & \text{when} \ \arg(\tau) \in (0,0.19) ,\\
N_-^{-1} h(\tau) & \text{when} \ \arg(\tau) \in (-\pi/2,0),
\end{cases}
\end{gather*}
that express equation~\eqref{eq.hv4} in the equivalent form
\begin{gather}
\label{eq.52was}
w^{(\s_j)}(\tau) \sim \Phih^{(\sigma_j)}(2\pi {\rm i}\tau),
\qquad \tau \to 0,
\end{gather}
when $\arg(\tau) \in (-\pi/2,0.19)\sm\{0\}$ and $j=1,2,3$.
Since the functions $h$ and $w$ are related by a linear transformation,
it follows that the state integral, the index and the quadratic identity
can be expressed in terms of the function $w$ as follows:
\begin{gather}
\label{eq.52quadwW}
0 = \sum_{j=1}^3 w^{(\s_j)} (\tau) w^{(\s_j)}(-\tau),
\\
2 Z_{5_2}(\tau) = \sum_{j=1}^3 w^{(\s_j)}(\tau-1)w^{(\s_j)}\big(\tau^{-1}-1\big) ,\label{eq.52swW}
\\
4 \Ind_{5_2}\big ({\rm e}^{2\pi {\rm i}\tau}\big)  =
w^{(\s_3)}(\tau)w^{(\s_3)}(-\tau)
-w^{(\s_1)}(\tau)w^{(\s_2)}(-\tau)
-w^{(\s_2)}(\tau)w^{(\s_1)}(-\tau) . \nonumber
\end{gather}
In terms of the $\Phi^{(\s_j)}$ series, equation~\eqref{eq.52quadwW}
and~\eqref{eq.52was} implies the quadratic identity
\begin{gather*}
\sum_{\sigma} \Phi^{(\s)}(x) \Phi^{(\s)}(-x) = 0 ,
\end{gather*}
(where we are summing over $\s \in \{\s_1,\s_2,\s_3\}$)
whereas equation~\eqref{eq.52swW} and ~\eqref{eq.52was} implies that
the expansion of $Z_{5_2}(\tau)$ around $\tau=1$ when
$\tau$ is given by~\eqref{eq.taux0} is a power series
\begin{gather*}
\sum_\sigma {\rm e}^{-C_\sigma}
\Phi^{(\s)}\biggl(\frac{2x}{1-x}\biggr)
\Phi^{(\s)}\biggl(-\frac{2x}{1+x}\biggr)
\end{gather*}
convergent when $|x|<1$. Here, $C_\s= \VC(\rho)/(2\pi {\rm i})$ where $\VC(\rho)$
is the complexified volume of the corresponding boundary parabolic
$\SL_2(\BC)$-representation $\rho$ of the fundamental group of the complement
of the $5_2$ knot.

\subsection{Higher level and weight spaces}
\label{sub.237}

In this section, we describe a new phenomenon, the level of a knot,
and examples where the weight spaces have higher multiplicity.
For the $(-2,3,7)$ pretzel knot, there are 6 pairs of $q$-series,
and the weight spaces are not one-dimensional; there are weights 0, 1 and 2
with dimensions~$1$,~$4$ and~$1$, respectively. The 6 pairs of $q$-series involve
power series in integer powers of $q^{1/2}$, meaning level $N=2$,
and so we should introduce the level of a knot, presumably the same as the
one coming from the periodicity of the degree of the colored Jones
polynomial~\cite{Ga:quasi, Ga:slope}. This $q^{1/2}$ will be upgraded to a whole
$\SL_2(\BZ)$ and $\Gamma(2)$ story in Section~\ref{sec.matrix}.
As an added complexity for the~$(-2,3,7)$ knot, the 6 asymptotic series
come in two Galois orbits of size 3 defined over the cubic field of
discriminant $-23$ (the trace field) and over the abelian field $\BQ(\cos(2\pi/7))$
of discriminant $49$. Moreover, the 3 complex volumes of the latter Galois
orbit are rational multiples of $\pi^2$.

To illustrate the new phenomenon, we begin by
introducing the 6 pairs of $q$-series for the $(-2,3,7)$ pretzel knot.
The state integral of the $(-2,3,7)$ pretzel knot was given
in~\cite[Appendix~B]{GK:evaluation}. Using the functional equation for Faddeev's quantum
dilogarithm~\cite[equation~(78)]{GK:evaluation}, and ignoring some prefactors, the
state integral is given by
\begin{gather}
\label{Psi237}
Z_{(-2,3,7)}(\tau)
 =
\biggl(\frac{q}{\tq}\biggr)^{-\frac{1}{24}}
\int_{\BR + {\rm i} c_b/2 + {\rm i} \varepsilon} \Phi_{\sqrt{\tau}}(x)^2 \Phi_{\sqrt{\tau}}(2x-c_b)
 {\rm e}^{-\pi {\rm i} (2x-c_b)^2} {\rm d}x, \qquad\tau\in\BC'
\end{gather}
with small positive $\varepsilon$, where $b=\sqrt{\tau}$ and
$c_b=\frac{{\rm i}}{2}\big(b+b^{-1}\big)$. Using the method of~\cite{GK:qseries}, we can express
the above state integral in terms of 6 $q$-series as follows.

\begin{Proposition}
\label{prop.237}
We have
\begin{align*}
\begin{split}
2 {\rm e}^{\frac{\pi {\rm i}}4}(q/\tq\bigr)^{1/24} Z_{(-2,3,7)}(\tau)
 ={}& - \frac{1}{2\tau} h_0(\tau) h_2\big(\tau^{-1}\big) \+ h_1(\tau) h_1\big(\tau^{-1}\big)
 - \frac\tau{2} h_2(\tau) h_0\big(\tau^{-1}\big)
\\ \notag
&{}
+\frac1{\tau} \bigl( h_3(\tau) h_4\big(\tau^{-1}\big) - h_4(\tau) h_3 \big(\tau^{-1}\big)
+ h_5(\tau) h_5\big(\tau^{-1}\big)\bigr)
\end{split}
\end{align*}
for $\tau \in \BC\sm\BR$, with the same convention as
in~\eqref{eq.hqqt}, but with $(\pm1)^j$ replaced by $(\pm1)^{\delta_j}$
where $(\delta_0,\dots,\delta_5)=(0,1,2,0,0,0)$ denotes the $\ve$-deformation degree
and where the $H^\pm_j(q)$ are power series in $q^{1/2}$ whose first few terms
are given by
\begin{gather}
 H^+_0(q)   =  1 + q^3 + 3 q^4 + 7 q^5 + 13 q^6 +\cdots, \nonumber
\\
 H^-_0(q)   =  1 + q^2 + 3 q^3 + 7 q^4 + 13 q^5 +\cdots, \nonumber
\\
 H^+_1(q)   =  1 - 4 q - 8 q^2 - 3 q^3 + 3 q^4 +\cdots, \nonumber
\\
 H^-_1(q)   =  1 - 4 q - 5 q^2 + q^3 + 7 q^4 +\cdots, \nonumber
\\
 H^+_2(q)   =  \frac{2}{3} - 6 q + 6 q^2 + \frac{242}{3} q^3 + 200 q^4 +\cdots, \nonumber
\\
 H^-_2(q) =  \frac{5}{6} - 10 q + \frac{17}{6} q^2 + \frac{141}{2} q^3 + \frac{971}{6} q^4 +\cdots, \nonumber
\\
 H^+_3(q)   =  q + 3 q^2 - 2 q^{5/2} + 8 q^3 - 8 q^{7/2} +\cdots, \nonumber
\\
 H^-_3(q)   =  q + 4 q^{3/2} + 9 q^2 + 18 q^{5/2} + 31 q^3 +\cdots, \nonumber
\\
 H^+_4(q)   =  1 + 4 q + 12 q^2 + 33 q^3 + 79 q^4 +\cdots, \nonumber
\\
 H^-_4(q) =  \frac{1}{4} - q + \frac{5}{4} q^2 - \frac{5}{4} q^3 + \frac{15}{4} q^4 +\cdots, \nonumber
\\
 H^+_5(q)   =  q + 3 q^2 + 2 q^{5/2} + 8 q^3 + 8 q^{7/2} +\cdots, \nonumber
\\
 H^-_5(q)   =  q - 4 q^{3/2} + 9 q^2 - 18 q^{5/2} + 31 q^3 +\cdots, \label{eq.237H}
\end{gather}
and whose precise definition and properties are given in Appendix~{\rm \ref{sub.q237}}.
\end{Proposition}

The vector space $\la H \ra$ spanned by $(H_0,\dots,H_5)$ has the
($\ve$-deformation) weight decomposition
\[
\la h \ra = W_0 \oplus W_1 \oplus W_2, \qquad W_0=\la H_0, H_3, H_4, H_5 \ra, \qquad
 W_1 = \la H_1 \ra, \qquad W_2 = \la H_2 \ra .
\]
There is a representation $\rho$ of $\SL_2(\BZ)$ on $\la H \ra$ which is
the identity on $W_1$ and $W_2$ and has kernel $\Gamma(2)$ on $W_0$. Thus,
the action of $\rho$ on $W_0$ comes from a representation $\rho'$ of the
quotient group $S_3=\Gamma/\Gamma(2)$. The latter decomposes as the direct sum
of the 2-dimensional irreducible representation of $S_3$ and two copies of
the trivial representation of $S_3$.

The index of the $(-2,3,7)$ pretzel knot is given by the following expression:
\begin{align*}
\Ind_{(-2,3,7)}(q) & = \sum_{k_1, k_2, k_3 \in \BZ}
(-q^{\frac{1}{2}})^{k_1 - 2 k_2} \ID(2 k_2, k_1 - 2 k_2 - k_3) \\
& \qquad\qquad\quad \times \ID(-k_1 + k_2, k_1 - 2 k_2) \ID(k_1 - 2 k_2 - 2 k_3, k_2) \\
&  =  1 - 8 q + 3 q^2 + 50 q^3 + 58 q^4 + 13 q^5 - 196 q^6 - 456 q^7
- \cdots.
\end{align*}

\begin{observation}
The relation with the index is given by
\begin{gather*}
\Ind_{(-2,3,7)}(q)  =  H^+_1(q) H^-_1(q)
\end{gather*}
and the following quadratic relation holds:
\begin{gather*}
 \frac{1}{2} H^+_0(q) H^-_2(q) - H^+_1(q) H^-_1(q) + \frac{1}{2} H^+_2(q) H^-_0(q)
\\
\qquad {}- H^+_3(q) H^-_3(q) + H^+_4(q) H^-_4(q) - H^+_5(q) H^-_5(q) =0 .
\end{gather*}
\end{observation}

Just in the case of the $4_1$ knot and the $5_2$ knots, the asymptotics of
$h_j(\tau)$ as $\tau\in \BC\sm\BR$ tends to zero in a ray are given
by a rational linear combination of the asymptotic series $\Phih^{(\s)}(h)$
that appear in the quantum modularity conjecture of the $(-2,3,7)$
knot~\cite{GZ:kashaev}. However, this knot has 6 boundary parabolic
$\SL_2(\BC)$ representations, arranged in two Galois orbits of size~3, one
defined over the trace field of the $(-2,3,7)$ pretzel knot \big(the cubic field
of discriminant~$-23$ generated by $\xi$ with $\xi^3-\xi^2+1=0$\big) and another
defined over the real abelian field $\BQ(2\cos(2\pi/7))$. Let
$\{\s_1,\s_2,\s_3\}$ denote the three embeddings of the trace field
corresponding to $\Im(\xi)<0$, $\Im(\xi)>0$ and $\Im(\xi)=0$, and
let $\{\s_4,\s_5,\s_6\}$ denote the three embeddings of
$\Q(\eta)$ with $\eta^3+\eta^2-2\eta-1=0$ (the abelian cubic field with
discriminant~$49$) into~$\BC$ given by sending~$\eta$ to $2\cos(2\pi/7)$,
$2\cos(4\pi/7)$ and $2\cos(6\pi/7)$, respectively.
When $\tau$ approaches zero in a fixed generic ray, the six asymptotic
series $\Phih^{(\s_j)}(h)$ have different growth rates, and the ordering of
the growth rates in each ray is dictated by Figure~\ref{f.237rates}. \looseness=1

\begin{figure}[t]\centering
\includegraphics[height=0.17\textheight]{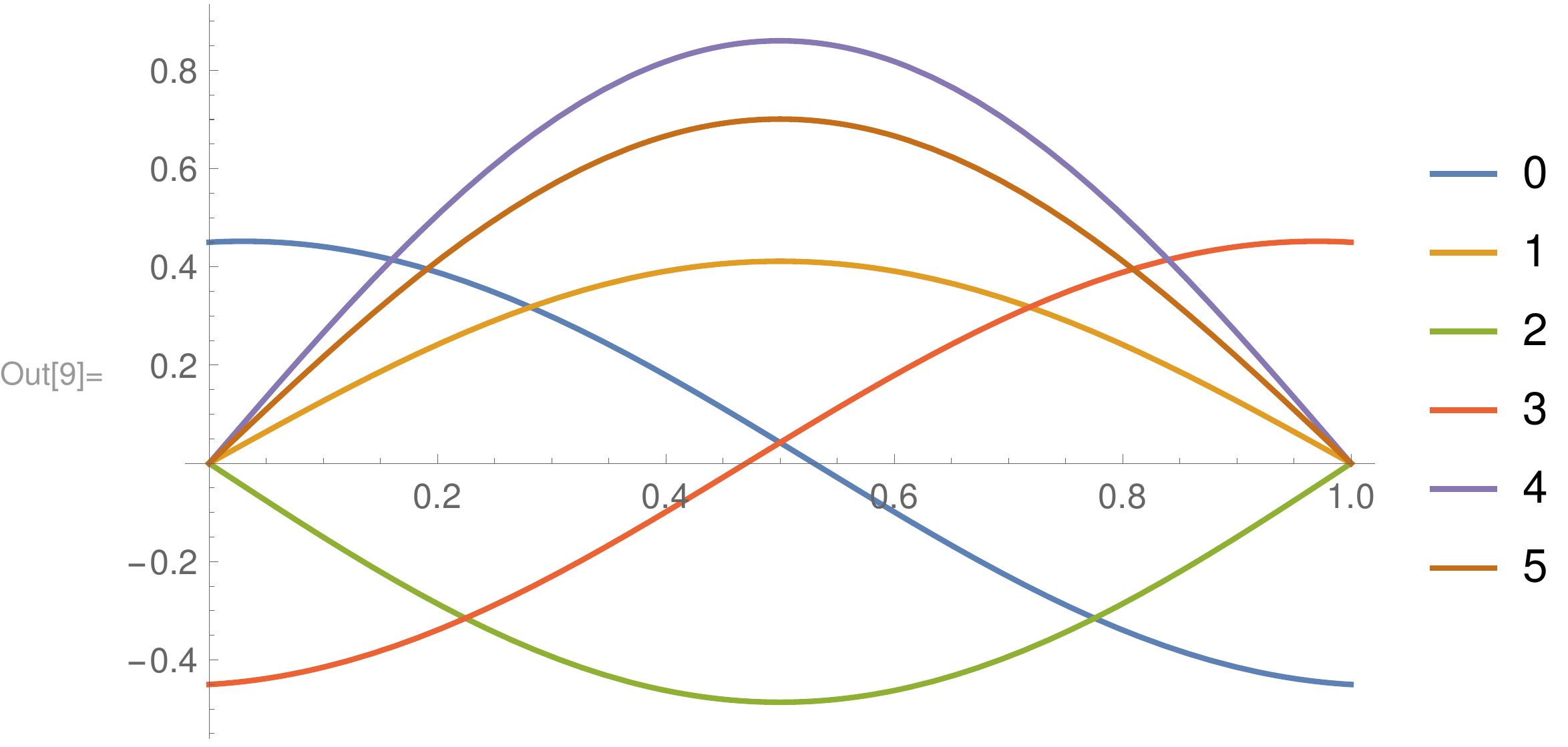}
\caption{A plot of the growth rates $\Re(\VC(\rho_j)/2\pi {\rm i}\tau)$
 of $w^{(j)}(x)$ for $j=0,\dots,5$ where $\arg(\tau)=\pi\th$ and
 $\theta \in (0,\pi)$. The two Galois orbits are $1,2,4$ and $0,3,5$ for the
 number fields of discriminant 49 and~-23. The branches cross at
 $0.$, $0.16$, $0.19$, $0.22$, $0.28$, $0.5$, $0.71$, $0.77$, $0.81$, $0.84$, $1$
 and partition the interval~$[0,1]$ in 10 sectors.}
\label{f.237rates}
\end{figure}

Let $\Phih_\a(h)=\big(\Phih^{(\s_j)}(h)\big)_{j=1}^6$ denote the vector of asymptotic
series, and let $h(\tau)=(h_j(\tau))_{j=0}^5$ denote the vector
of holomorphic functions on $\BC\sm\BR$ with weight $(-1,0,1,-1,-1,-1)$.
As before, if we let $X \to \infty$ in a fixed sector and $\g \in \SL_2(\BZ)$,
we can fit the asymptotic expansion of the vector $h|_\gamma(X)$ with the
asymptotic series $\Phih_\a(2\pi {\rm i}/(c X+d))$ after multiplication by a~matrix.
There is an additional subtlety which is absent in the case of the $4_1$ and $5_2$
knots, namely the fact that some of the $q$-series $H^{\pm}_j(q)$ are power series
in $q^{1/2}$, which implies that the functions~$h_j(\tau)$ are 2-periodic,
but not 1-periodic. This implies that
the matrices that determine the linear combinations depend on the cosets
of $\Gamma(2)$ in $\SL_2(\BZ)$.

\begin{observation}
 As $X \in \BC\sm\BR$ in a sector near the positive real
 axis and $X \to \infty$, we have
\begin{gather}
\label{eq.h237}
h|_\gamma(X) \sim \rho(\gamma)
\begin{pmatrix}
0 & \hphantom{-}1 & -1 & \hphantom{-}0 & -1 & -1/2 \\
0 & \hphantom{-}1 & \hphantom{-}1 & \hphantom{-}0 & \hphantom{-}0 & \hphantom{-}0 \\
0 & \hphantom{-}2/3 & -2/3 & \hphantom{-}0 & \hphantom{-}4/3 & \hphantom{-}1/6 \\
0 & -1 & \hphantom{-}1 & \hphantom{-}0 & \hphantom{-}1 & -1/2 \\
0 & \hphantom{-}0 & \hphantom{-}0 & -1/2 & -1 & \hphantom{-}0 \\
2 & \hphantom{-}0 & \hphantom{-}0 & -1/2 & -1 & \hphantom{-}0 \\
\end{pmatrix}
\Phih_\a\biggl(\frac{2\pi {\rm i}}{c X+d}\biggr)
\end{gather}
to all orders in $1/X$.
\end{observation}


Inverting the matrix in equation~\eqref{eq.h237}, allows one to define
holomorphic lifts $w^{(\s)}$ in $\BC\sm\BR$ of the asymptotic series
$\Phi^{(\s)}(h)$. This gives a practical method for computing the coefficients
of the 6 asymptotic series $\Phi^{(\s)}(h)$. Indeed, a numerical computation
of the series $w^{(\s)}$ at cusps and the Galois invariant of the series
$\Phi^{(\s)}(h)$ reduces the computation of their coefficient to the
recognition of rational numbers with prescribed denominators. We used this
method to compute~37 terms of the six $\Phi^{(\s)}(h)$, and to compare the
results with the asymptotics of the Kashaev invariant in~\cite{GZ:kashaev}.


\section[From vector-valued to matrix-valued q-series]{From vector-valued to matrix-valued $\boldsymbol{q}$-series}
\label{sec.matrix}

So far, we used the state integral of a knot to define a vector
of $q$-series for $|q| \neq 1$ whose asymptotics were found to be related to the
$r$-vector of asymptotic series of the knot from our earlier paper~\cite{GZ:kashaev}.
In this section, we report a recent discovery, descendants, which places the vector as
the first column of an invertible $r$ by $r$ matrix of $q$-series for $|q| \neq 1$.
It turns out that asymptotic series~\cite{GZ:kashaev}, $q$-series and state
integrals~\cite{GGM, GGM:peacock} all have descendants. We will explain the
notion of descendants in Section~\ref{sub.4.1} for the $4_1$ knot, where there will be
infinitely many descendants $G_0^{(m)}(q)$ and $G_1^{(m)}(q)$ (Laurent
series in $q$ with integer coefficients) with $m$ ranging over~$\Z$,
and then we will construct the matrix $Q(q)$ whose second column is
$\frac{1}{2}\big(q G_j^{(1)} - q^{-1} G^{(-1)}_j\big)$ for $j=0,1$. (We will explain in
Section~\ref{sub.cocycles} below why we choose this particular linear combination.)
In Section~\ref{sub.41dasy} we discuss the asymptotic properties of these descendants,
and in Section~\ref{sub.52blocks} we state the analogous results for the $5_2$~knot.

\subsection[Descendant q-series]{Descendant $\boldsymbol{q}$-series} \label{sub.4.1}

In this section, we will focus on the $4_1$ knot following the work of
the first author, Gu and Mari\~{n}o~\cite{GGM} (with detailed proofs provided
in~\cite[Section~3.1]{GGM:peacock}) but using a slightly
different notation. Consider the pair $G_0^{(m)}(q)$ and
$G_1^{(m)}(q)$ of $q$-series from~\cite{GGM} for integers $m$
\begin{gather*}
G^{(m)}_0(q) =\sum_{n=0}^\infty (-1)^n \frac{q^{n(n+1)/2+m n}}{(q;q)_n^2},
\\
G^{(m)}_1(q) =\sum_{n=0}^\infty (-1)^n \frac{q^{n(n+1)/2+m n}}{(q;q)_n^2}
\Bigg(2m+ \calE_1(q) + 2 \sum_{j=1}^n \frac{1+q^j}{1-q^j} \Bigg) ,
\end{gather*}
for $|q|<1$ and extended to $|q|>1$ by $G^{(m)}_j\big(q^{-1}\big) = (-1)^j G^{(m)}_j(q)$.
Observe that $G^{(0)}_j(q)=G_j(q)$ for $j=0,1$, with $G_0(q)$ and $G_1(q)$ given
in~\eqref{eq.3g} and~\eqref{eq.G}, respectively. Consider the matrix
\begin{equation*}
w_m(q) =
\begin{pmatrix} G^{(m)}_0(q) & G^{(m)}_1(q) \\ G^{(m+1)}_0(q) & G^{(m+1)}_1(q)
\end{pmatrix}, \qquad |q| \neq 1 .
\end{equation*}

The properties of these functions are given in~\cite[Section~3.1]{GGM:peacock}.

\begin{Theorem*}[\cite{GGM:peacock}]
The matrix $w_m(q)$ is a fundamental solution of the linear $q$-difference
equation%
\begin{equation}
\label{41qdiff}
y_{m+1}(q) -(2-q^m) y_m(q) + y_{m-1}(q)=0, \qquad m \in \BZ .
\end{equation}
It has constant determinant
 \begin{equation}
\label{det41}
\det(w_m(q))=2
\end{equation}
and satisfies the symmetry and orthogonality properties
\begin{align*}
& w_m\big(q^{-1}\big)= w_{-m}(q)
 \begin{pmatrix} 1 & \hphantom{-}0 \\ 0 & -1 \end{pmatrix} ,
\\
&\frac{1}{2} w_m(q)
\begin{pmatrix} 0 & 1 \\ 1 & 0 \end{pmatrix}
w_{m}\big(q^{-1}\big)^T=
\begin{pmatrix} \hphantom{-}0 & 1 \\ -1 & 0 \end{pmatrix} 
\end{align*}
for all integers $m$ and for $|q| \neq 1$.
\end{Theorem*}
The descendant series $G^{(m)}_j(q)$ arise from a factorization
of the ``descendant state integral''
\begin{equation*}
Z_{4_1,m,m'}(\tau) =
\int_{\BR+{\rm i} 0} \Phi_{\sqrt{\tau}}(v)^2 {\rm e}^{-\pi {\rm i} v^2
+ 2\pi(m \tau^{1/2} - m' \tau^{-1/2})v} {\rm d} v, \qquad
m, m' \in \BZ
\end{equation*}
introduced in~\cite{GGM}. This is a holomorphic function of $\tau \in \BC'$ that
coincides with $Z_{4_1}(\tau)$ when $m=m'=0$ and can be expressed bilinearly
in terms of $G^{(m)}_j(q)$ and $G^{(m')}_j(\tq)$ as follows~\cite[equation~(69)]{GGM}:
\begin{equation*}
 Z_{4_1,m,m'}(\tau)  =
 (-1)^{m-m'+1} \frac {\rm i}2 q^{\frac{m}{2}+\frac1{24}}\tq^{\frac{m'}2-\frac1{24}}
 \biggl(\sqrt{\tau} G^{(m')}_0(\tq) G^{(m)}_1(q) - \frac{1}{\sqrt{\tau}}
 G^{(m')}_1(\tq) G^{(m)}_0(q)\biggr) .
\end{equation*}
(Here $\widetilde q=\mathbf{e}(-1/\t)$ as usual.) This implies that the matrix-valued function
\begin{align}
 \label{W410-desc}
 W_{m,m'}(\tau)  =  \big(w_{m'}(\tq)^T\big)^{-1}
 \begin{pmatrix} 1/\t & 0 \\ 0 & 1 \end{pmatrix} w_m(q)^T ,
\end{align}
which is originally defined only for $\t\in\BC\sm\BR$, extends holomorphically to
$\tau \in \BC'$ for all integers~$m$ and $m'$. A similar story of descendants
for the $5_2$ knot was given in~\cite[Section~4.1k]{GGM:peacock}, and will
be reproduced in Section~\ref{sub.52blocks} below.

\subsection{The asymptotics of the descendants}
\label{sub.41dasy}

In~\cite{GZ:kashaev}, studying the refined quantum modularity conjecture for the
$4_1$ knot, we found a 2 by 2 matrix of asymptotic series
\begin{gather*}
\widehat{\bf \Phi}(x)=
\begin{pmatrix}
 \Phih(x) & \hphantom{-}\Psih(x) \\
 {\rm i} \Phih(-x) & -{\rm i} \Psih(-x)
\end{pmatrix},
\end{gather*}
where $\Psih(x)={\rm e}^{C/h} \Psi(x)$ where $\Psi(x)$ is the series
\begin{gather*}
\Psi(x) = \sum_{j=0}^\infty B_j x^j , \qquad
B_j =  {\rm i} \frac{\sqrt[4]3}{2} \biggl(\frac{1}{72\sqrt{-3}}\biggr)^j \frac{b_j}{j!}
\end{gather*}
with $b_j\in\BQ$, the first values being given by
\begin{center}
\def\arraystretch{1.3}
\begin{tabular}{|c|c|c|c|c|c|c|c|}\hline
$j$ & $0$ & $1$ & $2$ & $3$ & $4$ & $5$ & $6$ 
\\ \hline
$b_j$ & $-1$ & $37$ & $1511$ & $1211729/5$
& $407317963/5$
& $331484358355/7$
& $1471507944921541/35$
\\ \hline
\end{tabular}
\end{center}
Naturally, we looked into the asymptotics of its descendant holomorphic blocks.
Since any three consecutive are related by the recursion~\eqref{41qdiff}, so
are their asymptotics. For consistency, and for symmetry, we looked into the
asymptotics of the descendant holomorphic blocks for $m=-1,0,1$. Naturally,
we expected that the series $\Psih$ as well as the series $\Phih$ would
show up, and indeed we found the following
asymptotics for the matrix of $q$-series defined by
\begin{gather}
\label{Q41}
Q(\t) =  w_0(q)^T\begin{pmatrix}1&{-\tfrac12}\\0&\hphantom{-}1\end{pmatrix}  =
\begin{pmatrix}
G_0^{(0)}(q) & \frac12\bigl(G_0^{(1)}(q) - G^{(-1)}_0(q)\bigr)\\
G_1^{(0)}(q) & \frac12\bigl(G_1^{(1)}(q) - G^{(-1)}_1(q)\bigr)
\end{pmatrix},
\qquad q=\mathbf{e}(\t).
\end{gather}

\begin{observation}
\label{ob.Qmat}
As $\tau \to 0$ in the upper half-plane, we have
\[
\begin{pmatrix}
1/\sqrt{\tau} & 0 \\ 0 & \sqrt{\tau}
\end{pmatrix} Q(\tau)  \sim
\begin{pmatrix}
1 & -1 \\ 1 & \hphantom{-}1
\end{pmatrix} \widehat{\bf \Phi}(2\pi {\rm i} \tau) .
\]
\end{observation}

Note that equation~\eqref{det41} implies that $\det(Q(\tau))=2$
for all $\tau$, and combined with the above, it follows that
the function $\widehat{\bf \Phi}(x)$ satisfies
\begin{gather*}
\det(\widehat{\bf \Phi}(x))=1
\end{gather*}
as well as the orthogonality property
\begin{gather*}
\label{phiorthog}
\widehat{\bf \Phi}(-x) \widehat{\bf \Phi}(x)^t =
\begin{pmatrix} 0 & {\rm i} \\ {\rm i} & 0 \end{pmatrix} .
\end{gather*}

\subsection[The case of the 5\_2 knot]{The case of the $\boldsymbol{5_2}$ knot} 
\label{sub.52blocks}

Consider the linear $q$-difference equation
\begin{equation}
 \label{52qdiffp}
 y_{m}(q) -3y_{m+1}(q)+\big(3-q^{2+m}\big) y_{m+2}(q) - y_{m+3}(q)=0, \qquad m \in \BZ ,
\end{equation}
In~\cite[Section~3.2]{GGM:peacock}, it was shown that
it has a fundamental solution sets given by the columns of the
following matrix
\begin{gather}
\label{Jqred52}
w_{m}(q)  =  w_{m}^{5_2}(q)  =
\begin{pmatrix}
H^{(m)}_{0}(q) & H^{(m+1)}_{0}(q) & H^{(m+2)}_{0}(q) \\
H^{(m)}_{1}(q) & H^{(m+1)}_{1}(q) & H^{(m+2)}_{1}(q) \\
H^{(m)}_{2}(q) & H^{(m+1)}_{2}(q) & H^{(m+2)}_{2}(q)
\end{pmatrix},
\qquad m\in\Z,\qquad |q| \neq 1,
\end{gather}
where for $|q|<1$
\begin{gather*}
H^{(m)}_{0}(q)
 =\sum_{n=0}^\infty \frac{q^{n(n+1)+nm}}{(q;q)_n^3} ,
\\
 H^{(m)}_{1}(q)
 =\sum_{n=0}^\infty \frac{q^{n(n+1)+nm}}{(q;q)_n^3}
 \big(1+2n+m-3 \calE_1^{(n)}(q)\big) ,
 \\
 H^{(m)}_{2}(q)
 =\sum_{n=0}^\infty \frac{q^{n(n+1)+nm}}{(q;q)_n^3}
\bigg((1+2n+m-3 \calE_1^{(n)}(q))^2-3\calE_2^{(n)}(q)-\frac{1}{6}\calE_2(q)\bigg) ,
\end{gather*}
and
\begin{gather*}
 H^{(-m)}_{0}\big(q^{-1}\big)
 =\sum_{n=0}^\infty (-1)^n\frac{q^{\frac{1}{2}n(n+1)+nm}}{(q;q)_n^3} ,
\\
 H^{(-m)}_{1}\big(q^{-1}\big)
 =\sum_{n=0}^\infty (-1)^n\frac{q^{\frac{1}{2}n(n+1)+nm}}{(q;q)_n^3}
 \bigg(\frac{1}{2}+n+m-3\calE_1^{(n)}(q)\bigg) ,
 \\
 H^{(-m)}_{2}\big(q^{-1}\big)
 =\sum_{n=0}^\infty (-1)^n\frac{q^{\frac{1}{2}n(n+1)+nm}}{(q;q)_n^3}
 \bigg(\bigg(\frac{1}{2}+n+m-3\calE_1^{(n)}(q)\bigg)^2-3\calE_2^{(n)}(q)
 -\frac{1}{12}\calE_2(q)\bigg)
\end{gather*}
with $\calE^{(n)}_k(q)$ defined in equation~\eqref{eq.Emell} below.
Note that when $m=0$, $H_j^{(0)}(q^{\pm 1})=H^\pm_j(q)$ where $H^\pm_j(q)$ are the
six $q$-series of the $5_2$ knot~\eqref{eq.52deform} that appear in the
factorization of its state-integral.

\begin{Theorem*}[\cite{GGM:peacock}]
 The function $w_m(q)$ defined by~\eqref{Jqred52} is a fundamental solution of the
 linear $q$-difference equation~\eqref{52qdiffp} that has constant determinant
\begin{equation*}
 \det(w_m(q))=2 ,
\end{equation*}
satisfies
the orthogonality property
\begin{equation}
 \label{WWT52b}
 \frac{1}{2} w_{m-1}(q)
 \begin{pmatrix} 0 & 0 & 1 \\ 0&2&0\\ 1 & 0 &0 \end{pmatrix}
 w_{-m-1}\big(q^{-1}\big)^T =
 \begin{pmatrix} 1 & 0 & 0 \\ 0 & 0 & 1\\0 & 1 &  3-q^{m} \end{pmatrix}
\end{equation}
as well as
\begin{equation*}
 \frac{1}{2} w_m(q)
 \begin{pmatrix} 0 & 0 & 1 \\ 0&2&0\\ 1 & 0 &0 \end{pmatrix}
 w_{\ell}\big(q^{-1}\big)^T \in \SL\big(3,\BZ\big[q^\pm\big]\big)
\end{equation*}
for all integers $m$, $\ell$ and for $|q| \neq 1$.
\end{Theorem*}

The series $H^{(m)}(q)$ for $|q| \neq 1$ appear in the factorization of the
descendant state integral of the $5_2$ knot
\begin{equation*}
Z_{5_2,m,m'}(\tau) =
\int_{\BR+{\rm i} 0} \Phi_{\sqrt{\tau}}(v)^3 {\rm e}^{-2 \pi {\rm i} v^2
+ 2\pi(m \tau^{1/2} - m' \tau^{-1/2})v} {\rm d} v, \qquad
m, m' \in \BZ, \qquad \tau \in \BC'
\end{equation*}
of~\cite{GGM}. It is a holomorphic function of $\tau \in \BC'$ that coincides with
$Z_{5_2}(\tau)$ when $m=m'=0$ and can be expressed
bilinearly in terms of $H^{(m)}(q)$ as follows:
\begin{align}
Z_{5_2,m,m'}(\tau) ={}&
(-1)^{m-m'+1}\frac{{\rm e}^{\frac{\pi {\rm i}}{4}}}{2}
q^{\frac{m}{2}}\tq^{\frac{m'}{2}} \nonumber
\left(\frac{q}{\tq}\right)^{\frac{1}{8}} \\  &{}\times
\big(\tau h^{(m)}_2(\tau) h^{(m')}_0\big(\tau^{-1}\big) + 2h^{(m)}_1(\tau)
h^{(m')}_1\big(\tau^{-1}\big) \+ \frac1\tau h^{(m)}_0(\tau) h^{(m')}_2\big(\tau^{-1}\big) \big),\label{520-desc-fac}
\end{align}
where
\begin{gather*}
h^{(m)}_j(\tau) := (-1)^j H^{(m)}_j\big ({\rm e}^{2\pi {\rm i} \tau}\big), \qquad
\tau \in \BC\setminus\BR
\end{gather*}
for $j=0,1,2$ and $m \in \BZ$.
It follows that the matrix-valued function
\begin{align*}
 W_{m,m'}(\tau) = \big(w_{m'}(\tq)^T\big)^{-1}
 \begin{pmatrix} \tau^{-1} & 0 & 0 \\ 0 & 1 & 0 \\ 0 & 0 & \tau
 \end{pmatrix} w_m(q)^T
\end{align*}
defined for $\tau=\BC\sm\BR$, has entries given by the
descendant state integrals (up to multiplication by a prefactor
of~\eqref{520-desc-fac}) and hence extends to a holomorphic function
of $\tau \in \BC'$ for all integers~$m$ and $m'$. Using this for
$m=-1$ and $m'=0$ and the orthogonality relation~\eqref{WWT52b}, it
follows that we can express the Borel sums of $\Phi(\tau)$ in a region
$R$ in terms of descendant state integrals and hence, as holomorphic
functions of $\tau \in \BC'$ as follows.

\section{The matrix-valued cocycle of a knot}

In this section, we extend the observations of the previous sections to
matrix-valued analytic functions which naturally give rise to a cocycle on
on the set of matrix-valued piece-wise analytic functions on $\BP^1(\BR)$.
What's more, we conjecture (and in the case of the $4_1$, prove) that this
cocycle, restricted to the rational numbers, exactly agrees with the cocycle
of our previous work~\cite{GZ:kashaev}, which naturally binds the two works
together and naturally leads to the concept of a~matrix-valued holomorphic
quantum modular form.

\subsection{An equivariant state integral}

We return to the $4_1$ knot. The factorization of the state-integral~\eqref{Psi41}
given in equation~\eqref{eq.prop1} in terms of the pair $(g_0(\tau),g_1(\tau))$
motivates us to consider the following function:
\be
\label{eq.I41g}
Z_{4_1}(\g;\t)  =  \frac{{\rm i}}{2} (\tq/q)^{1/24}
 \big((c\tau+d)^{-1/2} g_0(\tau) g_1(\gamma(\tau)) -
(c\tau+d)^{1/2} g_1(\tau) g_0(\gamma(\tau)) \bigr)
\ee
for an element $\g$ of $\SL_2(\BZ)$ and for $\tau \in \BC\sm\BR$, where now
$\tq$ denotes~$\mathbf{e}(\g\t)$. A priori, this function is not defined for any
real value of the argument. However, experimentally (by looking at the asymptotics
of the function as we approach real points vertically) we found the following.

\begin{observation}
 \label{ob.41gamma}
 For every $\g=\sma abcd \in \SL_2(\BZ)$, the function
 $Z_{4_1}(\g;\t)$ extends to the cut plane
 $\BC_\g:=\BC\sm\{\tau \mid c\tau+d \leq 0 \}$.
\end{observation}

To explain and prove this observation, we introduced
an $\SL_2(\BZ)$-version of the state-integral using an $\SL_2(\BZ)$-version of
Faddeev's quantum dilogarithm \big(where the latter function corresponds to
$\g=\sma 0{-1}1{\hphantom{-}0}$\big) that satisfies a pentagon identity. The functional properties
of this quantum dilogarithm implies that the corresponding state-integral extends
on $\BC_\g$, and its factorization coincides, up to elementary factors, with the
function $Z_{4_1}(\g;\t)$ for the case of the $4_1$ knot. This is discussed in
current joint work with Kashaev~\cite{GKZ:modular}, where in particular a
proof of the above observation is given.

\subsection{A matrix-valued cocycle}

The state integral $Z_{4_1}(\t)$ is just one component of a $2\times2$ matrix
closely related to the matrix~$W_{0,0}(\t)$ defined in equation~\eqref{W410-desc},
and similarly the equivariant state integral~\eqref{eq.I41g}, up to elementary
factors, becomes just one component of a $2\times2$ matrix-valued function
\begin{gather}
\label{Wg}
W_\g(\t)  =  Q(\g\t)^{-1} \diag\bigl((c\t+d)^{-1},1\bigr)Q(\t),
 \qquad \t \in \C\sm \R .
\end{gather}
Observation~\ref{ob.41gamma} now generalizes to the statement that the function~$W_\g$
extends holomorphically from the upper and lower half-planes to~$\C_\g $.
Its restriction to~$\C\sm\R$ is a matrix-valued holomorphic cocycle there,
meaning that it satisfies
\begin{gather}
\label{eq.Wgg'}
W_{\gamma \gamma'}(\tau)  =  W_\gamma(\gamma'\tau) W_{\gamma'}(\tau)
\end{gather}
because the diagonal matrix appearing as the middle factor in~\eqref{Wg} is a
cocycle, so that the function $\g\mapsto W_\g$ is a ``twisted coboundary''. If $W_\g$
extended to the whole plane, then this cocycle property would automatically extend
to the real line by continuity. This doesn't quite work since $W_\g$ does not extend
to the whole real line, but only to a subset of it, namely the set of~$x$ with $cx+d>0$,
depending on~$\g$. To solve this problem, we pass from $\SL_2(\Z)$ to its quotient
$\PSL_2(\Z)=\SL_2(\Z)/\{\pm1\}$ and define a $\PSL_2(\Z)$-cocycle
$\bg\mapsto W^\R_{\bg}$ with values in the group of piecewise-analytic
invertible matrix-valued functions on~$\P^1(\R)$ by setting
\begin{gather}
\label{Wcoc}
W^\R_{\bg}(x)  =  W_\g(\t)\bigr|_{\t=x} 
\qquad\text{for} \quad cx+d>0,
\end{gather}
observing that for any element $\bg$ of $\PSL_2(\Z)$ and $x\in\R\sm\{-d/c\}$
we can lift~$\bg$ to a unique element $\g\in\SL_2(\Z)$ with $cx+d$ positive.
Of course the new cocycle on~$\P^1(\R)$ is no longer a~coboundary in any sense.
But this is a bonus rather than a defect, since non-trivial cohomology classes
are more interesting than trivial ones.

In the paper~\cite{GZ:kashaev} we had also found a cocycle on piecewise analytic
functions on~$\R$ with a~completely different definition, in terms of the
asymptotics near rational numbers of generalized Habiro-like functions.
The two cocycles turn out to agree, provably for the $4_1$ knot and conjecturally
in general. We discuss this next.

\subsection{The two cocycles agree}
\label{sub.cocycles}

We now show that the cocycle~\eqref{Wcoc} and the one from our prior
paper~\cite{GZ:kashaev} agree for the case of the $4_1$ knot.

We first recall from~\cite[Section~7.1]{GZ:kashaev} the periodic function
$ J=J^{(4_1)}$ on $\BQ$ defined by
\begin{gather}
\label{Q+41}
J(x)  =  \begin{pmatrix} J_{1,1}(x) & J_{1,2}(x) \\ J_{2,1}(x) & J_{2,2}(x)
\end{pmatrix},
\end{gather}
where
\begin{gather}
J_{1,1}(x)  =  \frac{1}{\sqrt{c}\sqrt[4]{3}} \sum_{Z^c = \z_6}
 \prod_{j=1}^c \bigl|1\m q^jZ\bigr|^{2j/c}, \nonumber\\
J_{2,1}(x)  =  \frac{{\rm i}}{\sqrt{c}\sqrt[4]{3}} \sum_{Z^c = \z_6^{-1}}
 \prod_{j=1}^c \bigl|1\m q^jZ\bigr|^{2j/c} ,\nonumber\\
J_{1,2}(x)  =  \frac{1}{2\sqrt{c} \sqrt[4]{3}}
\sum_{Z^c = \z_6} \bigl(Zq\m Z\i q\i\bigr) \prod_{j=1}^c \bigl|1\m q^jZ\bigr|^{2j/c},\nonumber
\\
J_{2,2}(x)  =  \frac{{\rm i}}{2\sqrt c \sqrt[4]{3}}
\sum_{Z^c = \z_6^{-1}} \bigl(Zq\m Z\i q\i\bigr)
\prod_{j=1}^c \bigl|1\m q^jZ\bigr|^{2j/c} \label{Q41abcd}
\end{gather}
with $q=\mathbf{e}(x)$ and $c=\text{denom}(x)$ being the denominator of $x$. (Actually,
the periodic
function defined in~\cite{GZ:kashaev}, and denoted there by~$\bJ=\bJ^{(4_1)}$, was a
$3\times3$ matrix with first column $(1 0 0)^T$ and bottom $2\times2$ piece~$J$,
but we will only need this part of it.)
The matrix $J$ defines a cocycle~\cite[Section~5]{GZ:kashaev}
\begin{gather}
\label{Wdef}
W^\Hab_\g(x)  =  J(\g x)^{-1} \diag\bigl ({\rm e}^{C\lambda_\g(x)}, {\rm e}^{-C\lambda_\g(x)}\bigr)   J(x) ,
\end{gather}
where $C$ is $1/2\pi$ times the volume of the figure~8 knot and
$\g\mapsto\lambda_\g$ is the~$\Q$-valued cocycle defined in equation~(24)
of~\cite{GZ:kashaev}
\begin{gather*}
\lambda_\g(x)  :=  \frac{1}{\den(x)^2\big(x-\g^{-1}(\infty)\big)}  =  \frac{c}{s(cr+ds)}
 =  \pm \frac c{\den(x)\den(\g x)} .
\end{gather*}
One of the main discoveries of~\cite{GZ:kashaev}, conjectural in general but
proved for the $4_1$~knot, is that this coboundary extends smoothly from
$\Q\sm\{-d/c\}$ to~$\BR\sm\{-d/c\}$. (Actually, in~\cite{GZ:kashaev} only a~somewhat
weaker statement was discussed, namely, that the function on~$\Q$ has a power
series to all orders in~$x-x_0$ as the argument $x$ tends to a fixed rational
number~$x_0$, with the stronger statement with smoothness, or even real-analyticity,
being mentioned there as an consequence of the results in the current paper.)

The next theorem links the cocycle of our
paper~\cite{GZ:kashaev} with the one of the current paper and explains the
bond between our two papers.

\begin{Theorem}
\label{thm.1}
The cocycles $W^\R$ and $W^\Hab$ coincide.
\end{Theorem}
Because $\SL_2(\BZ)$ is generated by $S=\sma 0{-1}1{\hphantom{-}0}$ and $T=\sma 1101$, and
both of the functions under consideration are cocycles and are trivial
on~$T$, and because both are continuous on ~$\R\sm\{0\}$, it is enough to
prove the equality
\begin{gather*}
W_S^\Hab(x)  =  W_S(x) \qquad\text{for} \quad x\in\Q^* .
\end{gather*}
The proof of this identity, given in Appendix~\ref{app.cocycles},
uses a ``factorization'' of state integrals at
positive rational points (i.e., a bilinear expression of a vector of functions
of $\tau \in \BQ$ and $-1/\tau$) established by Kashaev and the first
author~\cite{GK:evaluation}, similar to the ``factorization'' of state integrals
when $\tau \in \BC\sm\BR$ of the first author and Kashaev~\cite{GK:qseries}.
These two ``factorization'' properties of state integrals, one in the upper-half
plane and another in the positive rational numbers, are separate (in the sense that
we do not know how to deduce one from the other) but closely-related facts.


\subsection{Matrix-valued holomorphic quantum modular forms}
\label{sub.matrix}

We believe that the results we have been describing for the $4_1$~knot will
apply to all hyperbolic knots (possibly with the disclaimers given in the
introduction to~\cite{GZ:kashaev} about the behavior of character varieties
of general knots). Some part of the story, the matrix of ``descendant" functions
and the factorization formula~\eqref{520-desc-fac}, was carried out for the
$5_2$~knot in~\cite{GGM} and described in Section~\ref{sub.52blocks}, and
another part, the asymptotics (analogue of Observation~\ref{ob.Qmat}) was
carried out for the same knot in~\cite{GW:asy3D}. For the $(-2,3,7)$-pretzel
knot, only a part of the story, concerning what should be the upper left-hand
of the matrix~$W_S(\t)$ for this knot, was given in Section~\ref{sub.237}.
We have not done the corresponding calculations for any other knots, but the
expected pattern is clear and will be told here. These examples will also
lead to a new notion of ``matrix-valued holomorphic quantum modular forms"
which we expect will be of interest also in areas that are unrelated
to quantum topology.

To each hyperbolic knot we are going to assign various $r\times r$ matrices,
where $r$ is the number of non-trivial boundary parabolic $\SL_2(\C)$-representations.
(Some of them, and perhaps all, extend to square matrices of size~$r+1$ including
also the trivial representation, as discussed in~\cite{GZ:kashaev}
and~\cite{Wheeler:thesis},
but we will not go into this here. These larger matrices were denoted by boldface
letters there and we will use non-boldface names here to distinguish them.) Some of
these will be periodic functions (on either~$\Q$ or~$\C\sm\R$), but with the property
that the corresponding coboundaries lead to the {\it same} cocycle $W_\g$ with values
in the group of invertible matrices of piecewise analytic functions on~$\P^1(\R)$.
The periodic functions on~$\Q$ are either the generalized Habiro
functions~$\bJ^{(K)}(\a)$ or the related matrices of power series $\Phi(\a)$ of the
previous paper~\cite{GZ:kashaev}, whereas the matrix-valued functions in $\C\sm\R$
are the functions $Q=Q^K(\t)$ studied here. They have the following properties
and interrelations:
\begin{itemize}\itemsep=0pt
\item[$({\rm i})$] The matrix $Q=Q^K$ is a holomorphic and periodic in $\C\sm\R$ and meromorphic
at infinity, meaning that each of its entries is a power series in some rational
power of $q=\mathbf{e} (\t)$ in the upper half-plane and in $q^{-1}$ in the lower half-plane.
We also have ``weight'' $k=(k_1,\dots,k_r)\in\Z^r$ and a representation
$\rho\colon \SL_2(\BZ) \to \GL_r(\BC)$ which factors through $\Gamma(N)$ for some
integer~$N$ (called the {\it level} of the knot) which are compatible in the sense
that the map
\[
\g \;\mapsto\; j_\g(\t) := \rho(\g) \diag\big((c\t+d)^{k_i}\big)
\qquad\text{for} \quad\g=\sma abcd
\]
is a cocycle on~$\SL_2(\Z)$. (The representation~$\rho$ is a minor technical point
that arose in~\cite{GZ:kashaev} for the $(-2,3,7)$-pretzel knot but was trivial
for both the $4_1$ and $5_2$ knots and can be ignored.) The key property, which
is the one that says that $Q$ is a holomorphic quantum modular form, is that the
matrix-valued function
\begin{gather*}
W_\gamma(\tau) = Q(\g\t)^{-1} j_\g(\t) Q(\tau), \qquad \tau \in \BC\sm \BR
\end{gather*}
extends holomorphically from $\C\sm\R$ to $\C_\g$ for each~$\g\in\SL_2(\Z)$, just
as we saw above for the $4_1$~knot. This map automatically satisfies
equation~\eqref{eq.Wgg'}, and therefore leads to a~$\PSL_2(\Z)$-cocycle~$w^\R$,
with values in the ring of invertible piecewise analytic matrix-valued functions
on $\P^1(\R)$, by the same formula~\eqref{Wcoc} as before.

\item[$({\rm ii})$] Secondly, we associate to the knot~$K$ a collection $\a\mapsto\Phih_\a(h)$
of matrices, indexed by numbers~$\a\in\Q/\Z)$ (or equivalently, by roots of unity),
which are the generalized Habiro invariants whose existence was conjectured, and
in some cases extensively checked numerically, in~\cite{GZ:kashaev}.
The entries of these matrices are completed power series in an formal variable~$h$,
where ``completed" means that they belong to ${\rm e}^{v/h}\C[[h]]$ for some $v$, which
in fact will depend only the column of the matrix in which the entry lies and will
be the appropriate complexified hyperbolic volume. We think of $\Phi_a(h)$ as the
value of some formal function~$\Phih$ at $x=\a+{\rm i}\hbar$, defined in infinitesimal
neighborhoods of all rational points~$\a$. The group~$\PSL_2(\BZ)$ acts on the
space of such formal functions, so that we again get a~coboundary
$\Phih(\g x)^{-1}\Phi_\g(x)$, and this turns out to become a smooth function~$W^\Hab_\g(x)$
of~$\g$ and of a~real variable $x\in\R\sm\big\{\g^{-1}(\infty))\big\}$.
(For a more precise statement, see~\cite[equation~(78)]{GZ:kashaev}.) This new
function is then of course a cocycle, and the conjectural general statement is
that it simply coincides with~$W^\R$. The relation with what we said for the~$4_1$
knot in Section~\ref{sub.cocycles} is that, if we write the completed
power series-valued matrix $\Phih_a(h)$ as the product on the left of a true
power series-valued matrix $\Phi_a(h)$ by the diagonal matrix with entries
${\rm e}^{V_j/c^2h}$ ($j=1,\dots,r$, $c=$denom$(\a)$),
and then define $J(\a)$ to be the
constant term $\Phi_\a(0)$ of this matrix, then we have yet another coboundary
defined by the obvious analogue of equation~\eqref{Wdef}. The latter is now a
$\GL_r(\C)$-valued function on rational numbers, that again extends continuously
to the same smooth cocycle~$W_\g$ as before. It is this latter statement that
directly generalizes Theorem~\ref{thm.1} above, but the statement that we want
to emphasize is that the {\it same} cocycle $\g\mapsto W_\g$ trivializes (i.e.,
becomes a~coboundary) in each of three larger spaces than the space of piecewise
real-analytic functions on~$\P^1(\R)$ in which it is originally defined.
We can think of each of these trivializations \big(given by~$Q$,~$\Phih$ and~$J$\big)
as realizations of the same object in different spaces, similar to the various
realizations of motives in differently defined cohomology groups.

\item[$({\rm iii})$] Finally, and in some sense quite amazingly, the cocycle $W$ is not only
determined by the completed power series-valued matrix-valued function
$\alpha\mapsto\Phi_\a$ as its coboundary, but conversely determines this function
uniquely by the asymptotic property.
\begin{gather*}
W_\gamma(X)^{-1} \sim
\Phih_{\gamma(\infty)}\biggl(\frac{2 \pi {\rm i}}{c(cX+d)}\biggr) .
\end{gather*}
\end{itemize}

The matrices we have been discussing have a number of further interesting
properties, some of which we list in no particular order.

{\it Orthogonality.} There exists a matrix $B \in \GL_r(\BC)$ such that
\begin{gather}
\label{eq.AB}
Q(-\tau)^t B Q(\tau)  =  I .
\end{gather}

{\it $q$-holonomicity.} This property was discussed for the Habiro-like matrix
invariants in~\cite{GZ:kashaev}, while its $q$-series analogue, of which
equation~\eqref{41qdiff} is a special case, was the starting point
of~\cite{GW:modular}.

{\it Unimodularity.} In the cases that we have looked at, all of the matrices
we have been discussing were unimodular. We do not know whether to expect this
property in general.

{\it Bilinearity.} Property~\eqref{eq.AB} implies that $W$ can be expressed
bilinearly in terms of the entries of~$Q$ by
\[
W_\gamma(\tau)  =  (Q|_{\ve \gamma \ve})^t(\ve \tau) B Q(\tau),
\]
where $\ve = \left(\begin{smallmatrix} -1 & 0 \\ \hphantom{-}0 & 1 \end{smallmatrix}\right)$.

{\it Taylor series.} The cocycle property of $W_\gamma$ allows one to compute
the Taylor series expansion of the smooth function $W_\gamma$ at every rational
point and express them bilinearly in terms of the matrix $\Phih$ as was done
in~\cite[Proposition~5.2]{GZ:kashaev}.

\section{Final remarks}

In this paper, we discussed the properties of a 2 by 2 matrix $Q$ of periodic
functions on $\BC\sm\BR$ associated to the $4_1$ knot (see equation~\eqref{Q41}).
On the other hand, in our companion paper~\cite{GZ:kashaev}, we constructed
a 3 by 3 matrix $Q^+$ of periodic functions on $\BQ$ (see equation~\eqref{Q+41}).
Wheeler~\cite{Wheeler:thesis} has found an extension of our 2 by 2 matrix $Q$
(with one boring column $(1,0,0)^t$ and one interesting row) using the
$\ve$-deformation series~\eqref{eq.41deform}.

Another aspect of the matrix $Q$ of $q$-series associated to a knot appears
to be in connection to the resurgence, i.e., analytic continuation, of the
the factorially divergent series $\Phih(h)$ in the complex Borel plane.
In fact the matrix $Q$ appears to completely describe this problem of analytic
continuation as found by the first author and Gu and
Mari\~{n}o~\cite{GGM, GGM:peacock}. The so-called Stokes constants of the
analytic continuation problem are integers, multiplied by integer powers of
$\tq={\rm e}^{-2\pi {\rm i}/\tau}$ that assemble into power series with integer coefficients
which are none other than the matrix $Q(-1/\tau)$. This approach to resurgence
of asymptotic series is similar to the one proposed abstractly by
Kontsevich--Soibelman~\cite{MaximTalks,KS:stability,KS,KS:wall-crossing,KS:ar}.

It is clear from the data that is used to define a state integral that
the proposed holomorphic quantum modular forms are not only associated to knots,
but more generally to suitable half-symplectic matrices introduced
in~\cite{GZ:kashaev}, or alternatively to combinatorial gadgets often
called $K_2$ Lagrangians.

The proposed quantum holomorphic modular forms that appear here presumably
correspond to the partition functions $Z(h)$ and $\widehat{Z}(q)$ predicted by the
ongoing program of Gukov and collaborators~\cite{Gukov:largeN,Gukov-Manolescu, Gukov:resurgence, Gukov:BPS} for general 3-manifolds.

In the present paper we do not study the dependence of the invariants on
Jacobi variables, but postpone this for a later study. An example of such
invariants with the Jacobi variable corresponding to the holonomy of the
meridian of a knot complement was given in~\cite{GGM:peacock}.


\appendix

\section{Complements and proofs}

In this appendix, we provide proofs of some of the observations, in particular
regarding the $4_1$ knot, that were made in Sections~\ref{sec.how} and~\ref{sec.qandh}.

\subsection[q-series identities]{$\boldsymbol{q}$-series identities}
\label{sub.ZwegersIdentity}

We begin by giving the proof of the two identities of equation~\eqref{eq.3g},
as communicated to us by Sander Zwegers. We will use the identity~\eqref{eq.qxinf}
and
\begin{gather}
\label{eq.qnqm}
\frac{1}{(q)_m (q)_n}  = \sum_{\substack{r,s,t \geq 0 \\ r+s=m,s+t=n}}
\frac{q^{rt}}{(q)_r (q)_s (q)_t},
\end{gather}
which may be found for instance in~\cite{Za:DL}, where we abbreviate
$(q)_n=(q;q)_n$. If we sum over $m$ using~\eqref{eq.qxinf} we find
\begin{gather*}
\sum_{m,n \geq 0}(-1)^{m+n} \frac{q^{\frac{1}{2}m^2 + mn + \frac{1}{2}n^2
 + \frac{1}{2}m+\frac{1}{2}n}}{(q)_m(q)_n}\\
 \qquad{}
= \sum_{n \geq 0} (-1)^n \frac{q^{\frac{1}{2} n^2 + \frac{1}{2} n}}{(q)_n}
(q^{n+1})_\infty
=(q;q)_\infty
\sum_{n \geq 0} (-1)^n \frac{q^{\frac{1}{2} n^2 + \frac{1}{2} n}}{(q)_n^2} .
\end{gather*}
Equation~\eqref{eq.qnqm} with $m=n$ gives
\[
\frac{1}{(q)_n^2}  =
\sum_{\substack{r,s,t \geq 0 \\ r+s=n,s+t=n}}
\frac{q^{rt}}{(q)_r (q)_s (q)_t} =
\sum_{\substack{r,s \geq 0 \\ r+s=n}}
\frac{q^{r^2}}{(q)_r^2 (q)_s} ,
\]
and so
\[
\sum_{m,n \geq 0} (-1)^{m+n} \frac{q^{\frac{1}{2}m^2 + mn + \frac{1}{2}n^2
+ \frac{1}{2}m+\frac{1}{2}n}}{(q)_m(q)_n} =
(q;q)_\infty \sum_{n \geq 0} (-1)^n q^{\frac{1}{2} n^2 + \frac{1}{2} n}
\sum_{\substack{r,s \geq 0 \\ r+s=n}}
\frac{q^{r^2}}{(q)_r^2 (q)_s}
\]
which we can also write as
\[
(q;q)_\infty \sum_{r,s \geq 0}(-1)^{r+s}
\frac{q^{\frac{3}{2} r^2 + rs +\frac{1}{2}s^2 +\frac{1}{2}r+\frac{1}{2}s}
}{(q)_r^2 (q)_s} .
\]
Summing over $s$ and using~\eqref{eq.qxinf} with $x=q^r$, we get that this
equals to
\[
(q;q)_\infty \sum_{r \geq 0} (-1)^r
\frac{q^{\frac{3}{2} r^2 +\frac{1}{2}r}}{(q)_r^2} (q^{r+1})_\infty
=
(q;q)_\infty^2 \sum_{r \geq 0} (-1)^r
\frac{q^{\frac{3}{2} r^2 +\frac{1}{2}r}}{(q)_r^3} .
\]
This concludes the proof of~\eqref{eq.3g}.
\qed

\subsection{Asymptotics at roots of unity}

For the comparison of the results of this paper and those of~\cite{GZ:kashaev},
we need to understand the asymptotics of our $q$-series near roots of unity.
This is not the main theme of the paper and we will not go into detail, but as
an indication of the method we prove Observation~\ref{ob.1} giving the asymptotics
(to all orders) of the two q-series $G_0(q)$ and $G_1(q)$ associated to the $4_1$
knot at $q=1$. For this purpose we will use the formula for $G_0(q)$
given in the second part of equation~\eqref{eq.3g}.

To find the asymptotics of $G_0(q)$, we use the ``Meinardus trick''
as explained in \cite[pp.~54--55]{Za:DL}. This would work using either identity
in equation~\eqref{eq.3g}, but since the first would lead to a~double rather than
a single integral, we use only the second one.
From the second representation of $G_0(q)$ in equation~\eqref{eq.3g} and the standard
expansion
\[
\frac1{\ps x} \;:=\; \prod_{i=0}^\infty\frac1{1\m q^ix}
 =  \sum_{n=0}^\infty\frac{x^n}{(q;q)_n}, \qquad q, x\in\C,\qquad |q|<1,
\]
we get the integral representation
\begin{gather*}
\ps q G_0(q)  = \ct\biggl(\frac{\Th(x)}{\ps x^2}\biggr)
  =  \int_{{\rm i} \ve+\R/\Z}\frac{\Th(\mathbf{e}(u))}{\ps{\mathbf{e}(u)}^2}\, {\rm d}u,
\end{gather*}
where $\ve>0$ is a small and positive, ``$\ct$'' means ``constant term'' with respect
to $x$, and $\Th(x)$ is defined by
\begin{gather*}
\Th(x)  =  \sum_{n=-\infty}^\infty (-1)^n q^{\binom{n+1}2} x^{-n},
\qquad q, x\in\C,\qquad |q|<1 .
\end{gather*}
From the transformation law
$\th(\t,u)=\sqrt{{\rm i}/\t} \mathbf{e}\bigl(-u^2/{\rm i}\t\bigr)\th(-1/\t,u/\t)$ of the
Jacobi theta function $\th(\t,u)=\mathbf{e}(\t/8+u/2)\Th(\mathbf{e}(u))$ we get
\[
\Th(\mathbf{e}(u))  =  \sqrt{\frac {\rm i}\t}
\sum_{\lambda \in u-\frac{1+\t}2+\Z}\mathbf{e}\biggl(\frac{\lambda^2\ttau}2\biggr),
\qquad q=\mathbf{e}(\t),\qquad \ttau=-1/\t .
\]
Inserting this into the integral representations of~$G_0(q)$ and
unfolding in the usual way gives
\[
\sqrt{\frac\t {\rm i}} G_0(q)  =  \frac1{\ps q}
\int_{{\rm i}\ve+\R}\frac{\mathbf{e}\big(\frac\ttau2\big(u-\frac{1+\t}2\big)^2\big)}{\ps{\mathbf{e}(u)}^2} \, {\rm d}u .
\]
We now apply the method of stationary phase to this integral, deforming the
path of integration to pass through a point where the derivative of the integrand
vanishes and then expanding as a Gaussian integral around this point to get the
desired asymptotic expansion. We use the standard (and easy) expansion
\[
\frac1{(x;{\rm e}^{-h})_\infty}  =
\exp\biggl(\frac{\Li_2(x)}h + \frac12 \log\biggl(\frac1{1-x}\biggr)
\+\frac x{1-x} \frac h{12} + \O\big(h^2\big)\biggr), \qquad h\to0,
\]
where $\Li_2(x)$ is the dilogarithm function, to find that the logarithm of
the integrand has an asymptotic expansion of the form
$\sum_{n=-1}^\infty A_n(u)h^n$, where
\[
A_{-1}(u)  =  -2\pi^2\bigg(u-\h\bigg)^2 \+ 2 \Li_2(\mathbf{e}(u)) .
\]
The function $A_{-1}(u)$ has two local maxima at $u=\h\pm\frac13$.
A careful analysis of each of the local maxima, whose details we omit, reproduces
each of the two terms in the asymptotic expansion~\eqref{MainConj}. A similar
analysis can be done for the asymptotics of the series $G_1(q)$ at $q=1$
using equation~\eqref{eq.Gsimple}. All of this was sketched for $q=1$
($q={\rm e}^{-h}$, $h\searrow0$), however it can be extended to the case of $q=\z {\rm e}^{-h}$
following ideas similar to those discussed in~\cite{GZ:asymptotics}. Finally,
we mention that in principle the formulas we have given for~$5_2$ would allow
us to compute the asymptotics for this case too, but we have not done this.

\subsection[The two matrix-valued cocycles for the 4\_1 knot agree]{The two matrix-valued cocycles for the $\boldsymbol{4_1}$ knot agree}
\label{app.cocycles}

In this section, we give the proof of Theorem~\ref{thm.1}. Let us begin by
explaining the choice of matrix~$Q(q)$ of $q$-series for the $4_1$ knot
given in equation~\eqref{Q41}, using the matrix-valued function~$\bJ$ on
the rational numbers from~\cite[equation~(95)]{GZ:kashaev} whose first row is
$1$, $J_0(x)$ and $\frac{1}{2}(q J_{1}(q) -q^{-1} J_{-1}(q))$ when $q=\mathbf{e}(x)$,
where
\begin{gather*}
J_{m}(q) = \sum_{n=0}^\infty (q;q)_n \big(q^{-1};q^{-1}\big)_n q^{m n}
\end{gather*}
is a sequence of elements of the Habiro ring for integers $m$ that satisfies
the linear $q$-difference equation
\begin{gather}
\label{rec41}
J_{m+1}(q) -(2-q^m) J_{m}(q) + J_{m-1}(q) = 1, \qquad m \in \BZ .
\end{gather}
It follows that the first row of $\bJ$ is a basis for the $\BQ[q^\pm]$-module
spanned by $\{J_m(q) | m \in \BZ\}$. The recursion~\eqref{rec41} is
an inhomogeneous analogue of~\eqref{41qdiff} and the first row of $\bJ$
above explains the choice for the second column of the matrix~\eqref{Q41}.

Observe next that the elements of the matrix $J$ given in~\eqref{Q41abcd}
can be written in the form
\begin{gather}
J_{1,1}(x) = \frac{1}{\sqrt{c\sqrt{-3}}} \sum_{Z^c = \z_6}
\calD_q(Z) \calD_{q^{-1}}\big(Z^{-1}\big), \nonumber \\
J_{2,1}(x) = \frac{1}{\sqrt{c\sqrt{-3}}}
 {\rm i} \sum_{Z^c = \z_6^{-1}}
\calD_q(Z) \calD_{q^{-1}}\big(Z^{-1}\big), \nonumber \\
J_{1,2}(x) = \frac{1}{2\sqrt{c\sqrt{-3}}}
\sum_{Z^c = \z_6} \big(q^{n+1}-q^{-n-1}\big)
\calD_q(Z) \calD_{q^{-1}}\big(Z^{-1}\big), \nonumber \\
J_{2,2}(x) = \frac{1}{2\sqrt{c\sqrt{-3}}} {\rm i}
\sum_{Z^c = \z_6^{-1}} \big(q^{n+1}-q^{-n-1}\big)
\calD_q(Z) \calD_{q^{-1}}\big(Z^{-1}\big),\label{J41D}
\end{gather}
where $c=\den(x)$ and $q=\mathbf{e}(x)$ where $\calD_\z(x)$ is the renormalized version
of the cyclic quantum dilogarithm $D_\z(x)$ given by
\begin{gather*}
\calD_\z(x)  =  {\rm e}^{-1/2 s(a,c)} D_\z(x), \qquad
D_\z(x)  =  {\rm e}^{-1/2 s(a,c)} \exp \Biggl(\sum_{j=1}^{c-1} \frac{j}{c} \log\big(1-\z^j x\big)
\Biggr)
\end{gather*}
when $\z=\mathbf{e}(a/c)$, where $s(a,c)$ is the Dedekind sum~\cite{Rademacher} and
where the logarithm is the principal one away from the cut at the negative real
axis and equals to the average one on the cut. The cyclic quantum dilogarithm
appears in the expansion of Faddeev's quantum dilogarithm at roots of unity
(see for example~\cite{GK:evaluation, Kashaev:star}) and plays a key role in
the definition of the near units associated to elements of the Bloch
group~\cite{CGZ}.

With the notation of Section~\ref{sub.cocycles}, the $2 \times 2$ matrix
$J$ from~\eqref{Q41abcd} has determinant $1$, hence its inverse is given by
$J^{-1}=\sma {J_{2,2}}{-J_{1,2}}{-J_{2,1}}{J_{1,1}}$. Since
$\lambda_S(\a)=\tfrac{1}{\text{den}(\a)\text{num}(\a)}$~\cite[Section~3.1]{GZ:kashaev}
(where $\text{num}(\a)$ and $\den(\a)>0$ denote the numerator and the denominator
of a rational number), it follows that when $\g=S$, the cocycle~\eqref{Wdef} is
given by
\begin{align}
W_S^\Hab(\a) ={}&
\begin{pmatrix}
 \hphantom{-}J_{2,2}(-1/\a) & -J_{1,2}(-1/\a) \\ -J_{2,1}(-1/\a) & \hphantom{-}J_{1,1}(-1/\a)
\end{pmatrix}
\begin{pmatrix}
 {\rm e}^{\tfrac{C}{\text{den}(\a)\text{num}(\a)}} & 0 \\
 0 & {\rm e}^{-\tfrac{C}{\text{den}(\a)\text{num}(\a)}}
\end{pmatrix} \nonumber\\
&{}\times \begin{pmatrix}
 J_{1,1}(\a) & J_{1,2}(\a) \\ J_{2,1}(\a) & J_{2,2}(\a)
\end{pmatrix}\label{W41habS}
\end{align}
for $\a$ a positive rational number and $C  =  V/(2\pi)  =  0.32\dots$ with
$V$ as in~\eqref{eq.V41},
On the other hand, the $2\times 2$ matrix $Q(\tau)$ of equation~\eqref{Q41}
has determinant $1$, and when $\g=S$, the cocycle~$W^\BR_S$ of equation~\eqref{Wcoc} is given by
\begin{gather}
\label{W41S}
W^\BR_S(\t) =
\begin{pmatrix}
 \hphantom{-}Q_{2,2}(-1/\t) & -Q_{1,2}(-1/\t) \\ -Q_{2,1}(-1/\t) & \hphantom{-}Q_{1,1}(-1/\t)
\end{pmatrix}
\begin{pmatrix}
\t^{-1} & 0 \\ 0 & 1
\end{pmatrix}
\begin{pmatrix}
 Q_{1,1}(\t) & Q_{1,2}(\t) \\ Q_{2,1}(\a) & Q_{2,2}(\t)
\end{pmatrix} .
\end{gather}
This and the factorization~\eqref{eq.prop1} of the state-integral implies that
$W^\BR_S(\t)_{\s_2,\s_1}$ equals to the state-integral $Z_{4_1}(\tau)$, up to
multiplication by elementary factors that involve a $(q/\tq)^{1/24}$, a root
of unity, a rational number and a square root of $\t$.

To identify the two cocycles~\eqref{W41habS} and~\eqref{W41S}, we use a
``factorization'' property of state integrals at
positive rational points (i.e., a bilinear expression of a vector of functions of
$\tau \in \BQ$ and $-1/\tau$) of the first author and Kashaev~\cite{GK:evaluation},
similar to the ``factorization'' of state integrals when $\tau \in \BC\sm\BR$
of the first author and Kashaev~\cite{GK:qseries}. These two factorization properties
of the state-integral are related, but we not know how to deduce one from the other.
Note also that Theorem 1.1 of~\cite{GK:evaluation} proves the needed factorization
for all 1-dimensional state integrals at positive rational numbers, and this
covers the case of all three knots (namely, the $4_1$, $5_2$ and $(-2,3,7)$ pretzel
knots) and all of their descendant state-integrals of interest to us.



We need to show that for all positive rational numbers $\a$, we have
\begin{gather}
\label{W=W41}
W_S(\a)  =  W^\Hab_S(-\a) .
\end{gather}

We will focus on the equality of the $(\s_2,\s_1)$ entries in the above
equality, in which case $W_S(\a) \Tsim Z_{4_1}(\a)$, where $\Tsim$ means equality
up to elementary factors. The proof also matches those elementary factors, and
moreover works for the remaining three entries of the above equation, since they
are all given by descendant state-integrals.

We write $\a=M/N$ with $M$ and $N$ fixed coprime positive integers, and let
$C=V/(2\pi)$, where $V=2.02\dots$ is the volume of the $4_1$ knot.
In~\cite[Theorem 1.1]{GK:evaluation}, it was shown that
\begin{align}
Z_{4_1}(M/N) ={}& \z_{24 MN}^{5-6N+3(M+N+1)^2}
\bigl( {\rm e}^{-\frac{C}{MN}} \z_{MN}^{-1} P_{M/N}\big(z_+,\th^+_N,\th^+_M\big)G_{M,N}\big(\th^+_N,\th^+_M\big) \nonumber
\\  &{}
-{\rm e}^{\frac{C}{MN}} \z_{MN} P_{M/N}\big(z_-,\th^-_N,\th^-_M\big)G_{M,N}\big(\th^-_N,\th^-_M\big) \bigr),\label{Irat}
\end{align}
where
\begin{alignat*}{4}
&z_+= \mathbf{e}(1/6),\qquad && \th^+_N = \mathbf{e}(1/(6N)),\qquad && \th^+_M= \mathbf{e}(1/(6M)), &\\
&z_-= \mathbf{e}(5/6),\qquad && \th^-_N = \mathbf{e}(5/(6N)),\qquad && \th^-_M= \mathbf{e}(5/(6M)), &
\end{alignat*}
$D_\z(x)=\prod_{k=1}^{c-1} \big(1-\z^k x\big)^{k/c}$ where $\z$ is a $c$-th root
of unity, $\z_{N}=\mathbf{e}(1/N)$,
\begin{gather*}
P_{M/N}\big(z_+,\th^+_N,\th^+_M\big) =
\frac{1}{\sqrt{-3}} \frac{(1-z_+)^{2+1/M+1/N+1/(2 MN)}}{\big(1-\th^+_N\big)^2 \big(1-\th^+_M\big)^2
D_{\z_N^M}\big(\th^+_N\big)^2 D_{\z_M^N}\big(\th^+_M\big)^2}
\end{gather*}
and
\begin{align*}
G_{M,N}\big(\th^+_N,\th^+_M\big) ={}& \frac{1}{\sqrt{MN}}
\sum_{k=0}^{MN-1} \frac{1}{\big(\z_N^M \th^+_N; \z_N^M\big)_{Pk}
\big(\z_N^{-M} (\th^+_N)^{-1}; \z_N^{-M}\big)_{Pk}} \\ &{}\times
\frac{1}{\big(\z_M^N \th^+_M; \z_M^N\big)_{Qk}\big(\z_M^{-N} \big(\th^+_M\big)^{-1}; \z_M^{-N}\big)_{Qk}}
\end{align*}
and $P$ and $Q$ are integers with $MP+NQ=1$ ($G_{M,N}$ is independent of the choice
of $P$ and $Q$). Likewise, we can define $P_{M/N}\big(z_+,\th^+_N,\th^+_M\big)$ and
$G_{M,N}\big(\th^-_N,\th^-_M\big)$.
Define
\begin{gather*}
S_{\s_1}(\a) = \frac{1}{\sqrt{N\sqrt{-3}}}
\frac{1}{|D_{\z_\a}\big(\th^+_N\big)|^2} \sum_{k=1}^{N-1} \frac{1}{|\big(\z_\a \th^+_N;\z_\a\big)|^2},
\\
S_{\s_2}(\a) = \frac{1}{\sqrt{N\sqrt{-3}}}
\frac{1}{|D_{\z_\a}\big(\th^-_N\big)|^2} \sum_{k=1}^{N-1} \frac{1}{|\big(\z_\a \th^-_N;\z_\a\big)|^2},
\end{gather*}
where $N$ is the denominator of $\a$ and $\lambda(x)$ is the tweaking function
from~\cite[Section 3.1]{GZ:kashaev}. Moving the quantum factorials from the
denominator to the numerator and using equation~\eqref{J41D}, it follows that
\[
S_{\s_1}(\a)  =  {\rm e}^{\lambda(\a)C} J_{1,1}(\a),
\qquad S_{\s_2}(\a)  =  {\rm e}^{-\lambda(\a)C} J_{2,1}(\a) .
\]

\begin{Lemma}
Up to a prefactor, for every $\a \in \BQ$ with $\a>0$, we have
\begin{gather}
\label{IR}
Z_{4_1}(\a)  =  {\rm e}^{-C/(MN)}
J_{1,1}(-1/\a) J_{2,1}(\a) - {\rm e}^{C/(MN)}
J_{1,1}(\a) J_{2,1}\bigl(-\a^{-1}\bigr) .
\end{gather}
\end{Lemma}

\begin{proof}
Let $\a=M/N$ with $M$, $N$ coprime positive integers as before. As was observed
in~\cite{GK:evaluation}, when $M=1$ we can choose $P=1$ and $Q=0$ and then
$P_{1/N}$ and $G_{1,N}$ are independent of $\th^+_M$ and given by
\[
P_{1/N}\big(z_+,\th^+_N\big) =
\frac{1}{\sqrt{-3}} \frac{(1-z_+)^{1 + 3/(2N)}}{\big(1-\th^+_N\big)^2 D_{\z_N}\big(\th^+_N\big)^2 }
\]
and
\[
G_{1,N}\big(\th^+_N\big) = \frac{1}{\sqrt{N}}
\sum_{k=0}^{N-1} \frac{1}{\big(\z_N \th^+_N; \z_N\big)_{k}
\big(\z_N^{-1} \big(\th^+_N\big)^{-1}; \z_N^{-1}\big)_{k}} .
\]
Now, the proof of~\cite[Lemma~2.2]{GK:evaluation} implies the following factorization:
of $G_{M,N}$
\begin{gather*}
\label{Sfac}
G_{M,N}\big(\th^+_N,\th^+_M\big) = G_{1,N} \big(\th^+_N\big) G_{M,1} \big(\th^+_M\big),
\end{gather*}
which was unfortunately not stated explicitly in~\cite{GK:evaluation}. The
key observation, using the notation of~\cite[Lemma~2.2]{GK:evaluation}, is that
in the context of~\cite[Theorem~1.1]{GK:evaluation}, we have $g^+_N(x;q_+)=g^-_N(y;q_-)=
(-x)^A(1-x)^B=1$, which implies that $g^+_{k+N}(x_+;q_+)=g^+_{k}(x_+;q_+)$ and
$g^-_{k+N}(x_-;q_-)=g^-_{k}(x_-;q_-)$. This, together with equation~\eqref{Irat} and
the fact that the tweaking function $\lambda$ satisfies
$\lambda(x)-\lambda(-1/x)=1/(\text{den}(x)\text{num}(x))$ (where $\text{num}(x)$
and $\text{den}(x)>0$ denotes the numerator and denominator of a rational number)
concludes the proof of the lemma.
\end{proof}

On the other hand, equation~\eqref{W41habS} implies that
\begin{gather}
\label{WS41v}
W^\Hab_S(\a)_{\s_2,\s_1}= {\rm e}^{-C/(MN)} J_{1,1}(-1/\a) J_{2,1}(\a) -
{\rm e}^{C/(MN)} J_{1,1}(\a) J_{2,1} \bigl(-\a^{-1}\bigr) .
\end{gather}
This and the previous lemma completes the proof of the $(\s_2,\s_1)$-entry
of~\eqref{W=W41}.
\qed

We may say that equations~\eqref{WS41v} and~\eqref{IR} syntactically agree.

More generally, in~\cite{GK:evaluation} the evaluation of the
1-dimensional state integrals
\[
Z_{A,B}(\tau)  =  \int_{\BR+ {\rm i} \epsilon} \Phi_\bb(x)^B {\rm e}^{-A \pi {\rm i} x^2} {\rm d}x, \qquad
\tau=\bb^2 \in \BC'
\]
at rational points was given, where $A$ and $B$ are integers with $B > A > 0$.
Following the notation of~\cite[Section~1.3]{GK:evaluation}, fix a pair of coprime
positive integers $M$ and $N$ and define
\begin{equation*}
\mathsf{b}= \sqrt{M/N}, \qquad \mathsf{s}= \sqrt{MN}
\end{equation*}
and
\begin{equation*}
\operatorname{g}(z)=
(-z)^{A} (1-z)^{-B} \in \BQ\big[z^{\pm 1}\big]
\end{equation*}
and
\begin{equation*}
\calS =\big\{ w \mid
\operatorname{g}\big ({\rm e}^{2\pi\mathsf{s}w}\big)=1,\,
 0 < \mathsf{s} \operatorname{Im}(w)-\lambda < 1 \big\},
\end{equation*}
where $\lambda$ is a generic real number such that
\begin{equation*}
-(M+N)/2 < \lambda < 0 .
\end{equation*}
Note that if $w \in\calS$, then ${\rm e}^{2\pi\mathsf{s}w}$ is an algebraic
number with a fixed choice of $N$ and $M$-th roots.
For $\a=M/N$ and $w \in \calS$, define
\begin{gather*}
S_w(\a) = \frac{(1-z)^{\frac{B}{2N} + \lambda(\a)\frac{B}{4}}}{\sqrt{N z g'(z)}}
\frac{{\rm e}^{-\frac{\lambda(\a) B}{2\pi {\rm i}}
 R(z)}}{D_{\z_\a}(\th^+)^B} \sum_{ k \bmod N} \bigl(-\th^+\bigr)^{A k}
\frac{\z_\a^{\frac{A}{2}k(k+1)}}{(\z_\a \th^+;\z_a)_k^B},
\end{gather*}
where $R(z)$ is the Rogers dilogarithm, and let $S(\a)=(S_w(\a))_{w \in \calS}$
and $S^\text{op}(\a)=(S_{\overline{w}}(\a))_{w \in \calS}$.
Then, using the factorization~\eqref{Sfac} and~\cite[Theorem 1.1]{GK:evaluation},
takes the form
\begin{gather*}
Z_{A,B}(\a) = p_\a S^\text{op}(-\a)^t S\big(\a^{-1}\big) = p_\a \sum_{w \in \calS}
S_w(-\a) S_{\overline{w}}\big(\a^{-1}\big),
\end{gather*}
where $ p_\a=\mathbf{e}\big((B+3A(M+N+1)^2-6MN\big)/(24MN))$.

\subsection[The six q-series for the 5\_2 knot]{The six $\boldsymbol{q}$-series for the $\boldsymbol{5_2}$ knot}
\label{sub.q52}

The six $q$-series of the $5_2$ knot are given by $q$-hypergeometric
sums and their $\ve$-deformations (see~\eqref{eq.52deform}),
and this allows for an efficient recursive computation of their coefficients
similar to the one for the $4_1$ knot given in~\eqref{eq.G01coeff}.
Explicitly, the six $q$-series are given by
\begin{gather}
\label{eq.52hHNahm}
H^+_j(q) = \sum_{m=0}^\infty t_m(q) p_m^{(j)}(q) , \qquad
H^-_j(q) = \sum_{m=0}^\infty T_m(q) P_m^{(j)}(q),
\qquad j=0, 1, 2
\end{gather}
with
\begin{gather}
\label{tT52}
t_m(q)  =  \frac{q^{m(m+1)}}{(q;q)^3_m} , \qquad
T_m(q)  =  \frac{(-1)^mq^{m(m+1)/2}}{(q;q)_m^{ 3}} ,
\end{gather}
and
\begin{gather*}
   p_m^{(0)}(q) =1 , \qquad p_m^{(1)}(q) =\frac{1+3\E_1(q)}4
 +\sum_{j=1}^m\frac{2+q^j}{1-q^j} ,\\\
 p_m^{(2)}(q) = p_m^{(1)}(q)^2 - \frac{3+\calE_2(q)}{24}
 +\sum_{j=1}^m\frac{3q^j}{(1-q^j)^2} , \\
   P_m^{(0)}(q) =1 , \qquad P_m^{(1)}(q) = \frac{3\E_1(q)-1}4
 +\sum_{j=1}^m\frac{1+2q^j}{1-q^j} ,\\
 P_m^{(2)}(q) = P_m^{(1)}(q)^2- \frac{\calE_2(q)-3}{24}
 +\sum_{j=1}^m\frac{3q^j}{(1-q^j)^2} .
\end{gather*}
Here $\calE_1(q)$ and $\calE_2(q)$ are the weight~1 and weight~2 Eisenstein series
defined by~\eqref{eq.ei1} and $\calE_2(q) = 1-24 \sum_{n \geq 1} \frac{q^n}{(1-q^n)^2}$,
respectively.
Since each of $t_m$, $T_m$, $p_m^{(j)}$ and $P_m^{(j)}$ can be obtained from
its predecessor in just $\text O(1)$ operations, we can use the
formulas~\eqref{eq.52hHNahm} to compute a several thousand coefficients of
$H_j^\pm$ efficiently.

\begin{Remark}
Our notation $\big(p^{(0)}_m, p^{(1)}_m, p^{(2)}_m\big)$,
$\big(P^{(0)}_m, P^{(1)}_m, P^{(2)}_m\big)$, $\big(H^+_0,H^+_1,H^+_2\big)$,
 $(H^-_0,H^-_1,\allowbreak H^-_2)$  coincide with the quantities denoted $(p_{3,m},p_{2,m},p_{1,m})$,
$\big(P_{1,m},\frac12P_{2,m},-P_{3,m}\big)$, $\big(g_3,g_2,g_1+\frac16\calE_2 g_3\big)$,
$\big(G_1,\frac12G_2,-G_3-\frac16\calE_2 G_1\big)$
in \cite[Section~1.4]{GK:qseries}. Our formula~\eqref{eq.52} matches with
\cite[Corollary~1.8]{GK:qseries} using the above translation of notation combined with the
quasi-modularity property
\[
E_2(\ti\tau)=\tau^2 E_2(\tau) - \frac{6 {\rm i} \tau}{\pi}
\]
of $E_2(\t):=\calE_2 \big({\rm e}^{2\pi {\rm i}\t}\big)$ (see~\cite[Proposition 6]{Za:123}). Note also
that if we define the functions $\E_1(q)$ and $\E_2(q)$ for~$|q|>1$ by
$\E_k(q)=-\E_k\big(q^{-1}\big)$ ($k=1,2)$, then $T_m(q)=t_m\big(q^{-1}\big)$ and
$P_m^{(j)}(q)=(-1)^jp_m^{(j)}\big(q^{-1}\big)$. It follows that both of the above
$q$-hypergeometric formulas are convergent for $|q| \neq 1$ and that they are
related by
\begin{gather*}
H^{-}_j(q)=(-1)^j H^{+}_j\big(q^{-1}\big), \qquad j=0,1,2 .
\end{gather*}
\end{Remark}

\subsection[epsilon-deformation and the factorization of the state]{$\boldsymbol{\ve}$-deformation and the factorization of the state
integral}
\label{sub.estate}

In this section, we comment further between a connection of the factorization
proof of state integrals taken from~\cite{GK:qseries} and the $\ve$-deformations
of difference/differential equations. Whereas only finitely many $q$-series
appear in the factorization of a state integral (via the residue theorem),
their $\ve$-deformation leads to a sequence of $q$-series that contains further
information, as Wheeler has recently found out~\cite{Wheeler:thesis}.
Consider the one-dimensional state integral
\begin{gather*}
\label{DefIAB}
\Psi_{A,A+\ti A}(\tau)
 = \int_{\BR + {\rm i} \varepsilon} \Phi_{\sqrt{\tau}}(x)^{A+\ti A}
 {\rm e}^{-A \pi {\rm i} x^2} {\rm d}x, \qquad z\in\BC'
\end{gather*}
from~\cite{GK:qseries} (which was denoted $\calI_{A,A+\ti A}(\sqrt z)$
in~ ibid) for positive integers $A, \tilde A >0$. Let us briefly recall
its factorization following~\cite{GK:qseries} and their notation.
Since $A, \ti A >0$, it follows that the integral is absolutely convergent.
The idea \big(when $\tau=b^{1/2}$ is in the upper half-plane\big) is to move the contour
of integration upwards, and collect the residues from the poles of the integrand.
The quantum dilogarithm is a meromorphic function of $x$ with poles given by
\begin{gather*}
x_{m,n}=c_b + {\rm i} b m + {\rm i} b^{-1}n =
{\rm i} b\biggl(m+\frac{1}{2}\biggr) +
{\rm i} b^{-1}\biggl(n+\frac{1}{2}\biggr), \qquad m, n \in \BN .
\end{gather*}
The quasi-periodicity of $\Phi_b(x)$
\begin{align*}
&\frac{\Phi_b(x+c_b+{\rm i}b)}{\Phi_b(x+c_b)}= \frac{1}{1-q {\rm e}^{2 \pi b x}},
\\
&\frac{\Phi_b\big(x+c_b+{\rm i}b^{-1}\big)}{\Phi_b(x+c_b)}=
\frac{1}{1-\tq^{-1} {\rm e}^{2 \pi b^{-1} x}}
= -\frac{\tq {\rm e}^{-2 \pi b^{-1} x}}{1-\tq {\rm e}^{-2 \pi b^{-1} x}}
\end{align*}
implies that
\begin{align}
\notag
\Phi_b(x+ x_{m,n})& = \Phi_b(x+c_b)
\frac{1}{{\rm i}(q {\rm e}^{2 \pi b x};q{\rm i})_m}
\frac{(-1)^n \tq^{\frac{n(n+1)}{2}} {\rm e}^{-2 \pi b^{-1} x n }}{
{\rm i}(\tq {\rm e}^{-2 \pi b^{-1} x};\tq{\rm i})_n} \\
\label{eq.Phixmn}
&  = \Phi_b(x+c_b)
\frac{1}{(q {\rm e}^\ve;q)_m}
\frac{(-1)^n \tq^{\frac{n(n+1)}{2}} {\rm e}^{\tv n }}{
\big(\tq {\rm e}^{\tv};\tq\big)_n} ,
\end{align}
where $q={\rm e}^{2 \pi {\rm i} b^2}$, $\tq={\rm e}^{-2 \pi {\rm i} b^{-2}}$,
$\ve =2 \pi b x$ and $\tv=-2 \pi b^{-1} x$.
Moreover, the formula
$\Phi_b(x)  =  \big({\rm e}^{2 \pi b (x+c_b)};q\big)_\infty/ \big({\rm e}^{2 \pi b^{-1} (x-c_b)};\tq\big)_\infty
$
implies that
\begin{align}
\label{eq.phicb}
&\Phi_b(x+c_b) = \frac{1}{1-{\rm e}^{2 \pi b^{-1}x}}
\frac{\big(q {\rm e}^{2 \pi b x};q\big)_\infty}{\big(\tq {\rm e}^{2 \pi b^{-1} x};\tq\big)_\infty}
 =  \frac{1}{1-{\rm e}^{-\tv}}
\frac{(q {\rm e}^\ve;q)_\infty}{(\tq {\rm e}^{-\tv};\tq)_\infty} .
\end{align}
Combining, we find a product that decouples
\begin{align}
\label{eq.phimn}
\Phi_b(x+ x_{m,n})&  =  \frac{(q {\rm e}^\ve;q)_\infty}{(\tq {\rm e}^{-\tv};\tq)_\infty}
\frac{1}{{\rm e}^{\tv/2}-{\rm e}^{-\tv/2}}
\frac{(-1)^n \tq^{\frac{n(n+1)}{2}}
{\rm e}^{(n+\frac{1}{2})\tv}}{(q {\rm e}^\ve;q)_m (\tq {\rm e}^{\tv};\tq)_n}
\end{align}
and an exponential that also decouples
\begin{align}
\label{eq.expxmn}
{\rm e}^{-\pi {\rm i} (x+x_{m,n})^2}
&= {\rm e}^{\frac{\pi {\rm i}}{2}} \biggl(\frac{q}{\tq}\biggr)^{\frac{1}{8}}
q^{\frac{m(m+1)}{2}} \tq^{-\frac{n(n+1)}{2}} (-1)^{m+n}
{\rm e}^{{\rm i} \frac{\ve\tv}{4\pi} + \ve(m+\frac{1}{2})-\tv(n+\frac{1}{2})} .
\end{align}
A similar computation works when $\Im(\tau)<0$, and the result is
the following.

\begin{theoremx}[\cite{GK:qseries}]
For $\tau \in \BC\sm\BR$, we have
\begin{align*}
\Psi_{A,A+\ti A}(\tau)={}& 2 \pi {\rm i} {\rm e}^{\frac{\pi {\rm i} A}{2}}
\biggl(\frac{q}{\tq}\biggr)^{\frac{A}{8}} \\
&{} \times
\mathrm{Res}_{x=0} \biggl( F_{A,\ti A}(q,\ve) F_{\ti A, A}(\tq,\tv)
{\rm e}^{{\rm i} A \frac{\ve\tv}{4\pi}} 
\frac{(q {\rm e}^\ve;q)_\infty^{A+\ti A}}{(\tq {\rm e}^{-\tv}; \tq)_\infty^{A+\ti A}}
\frac{1}{ ({\rm e}^{\tv/2}-{\rm e}^{-\tv/2})^{A+\ti A}} 
\biggr),
\notag
\end{align*}
where
\begin{align*}
F_{A,\ti A}(q,\ve)& =  \sum_{m=0}^\infty (-1)^{A m}
\frac{q^{A \frac{m(m+1)}{2}} {\rm e}^{A \ve (m+\frac{1}{2})}}{(q {\rm e}^\ve;q)_m^{A + \ti A}} .
\end{align*}
\end{theoremx}
We can think of $F_{A,\ti A}(q,\ve)$ as a function of a Jacobi variable $\ve$, or
as a power series in $\ve$ with coefficients rational functions in $q$ that
can be computed by expanding $(q {\rm e}^\ve;q)_m$ as a power series in
$\ve$. To do so, recall the $q$-series $\calE^{(m)}_\ell(q)$ \cite[equation~(29)]{GK:qseries}
\begin{gather}
\label{eq.Emell}
\calE^{(m)}_\ell(q) = \sum_{s=1}^\infty  s^{\ell-1} \frac{q^{s(m+1)}}{1-q^s}
= \sum_{s \geq 1, n >m} s^{\ell-1} q^{s n}
=\sum_{s \geq 1} \! \sigma_{\ell-1}^{(m)} q^s, \quad
\sigma_{\ell-1}^{(m)}= \sum_{d|s, s/d>m} \! d^{\ell-1}
.
\end{gather}
 The next lemma is contained
in~\cite[Proposition~2.2]{GK:qseries}. For completeness, we state it and prove it here.

\begin{Lemma}
We have
\begin{gather*}
\frac{1}{(q {\rm e}^\ve;q)_m}  =  \frac{1}{(q;q)_m} \cdot
\frac{(q;q)_\infty}{(q {\rm e}^\ve;q)_\infty} \cdot
\exp\Biggl( -\sum_{\ell=1}^\infty \calE^{(m)}_\ell(q) \frac{\ve^\ell}{\ell!}
\Biggr) .
\end{gather*}
\end{Lemma}

\begin{proof}
Using $(qx;q)_\infty = (qx;q)_\infty/\big(q^{m+1}x;q\big)_\infty$, it follows that
\begin{align*}
\frac{1}{(q {\rm e}^\ve;q)_m}  =  \frac{1}{(q;q)_m} \frac{(q;q)_m}{(q {\rm e}^\ve;q)_m}  =
\frac{1}{(q;q)_m} \cdot
\frac{(q;q)_\infty}{(q {\rm e}^\ve;q)_\infty} \cdot
\frac{\big(q^{m+1} {\rm e}^\ve;q\big)_\infty}{(q {\rm e}^\ve;q)_\infty} .
\end{align*}
Finally,
\begin{align*}
\frac{\big(q^{m+1} {\rm e}^\ve;q\big)_\infty}{(q {\rm e}^\ve;q)_\infty} & =
\prod_{n >m} \frac{1-q^n {\rm e}^\ve}{1-q^n}
\end{align*}
which implies that
\begin{align*}
\log \biggl( \frac{\big(q^{m+1} {\rm e}^\ve;q\big)_\infty}{(q {\rm e}^\ve;q)_\infty} \biggr) &=
\sum_{n>m} \sum_{j \geq 1} \frac{1}{j} q^{nj} \big(1-{\rm e}^{j \ve}\big)
= - \sum_{\ell=1}^\infty \frac{\ve^\ell}{\ell!}
\sum_{j >0} j^{\ell-1} \sum_{n >m} q^{n j} .
\end{align*}
The result follows by summing the geometric series in $n>m$.
\end{proof}

Using the above method, we can sketch a proof of Proposition~\ref{prop.237}
which expresses the state integral of the $(-2,3,7)$ pretzel knot as a
sum of products of $q$-series and $\tq$-series.

\subsection[The twelve q-series for the (-2,3,7) pretzel knot]{The twelve $\boldsymbol{q}$-series for the $\boldsymbol{(-2,3,7)}$ pretzel knot}
\label{sub.q237}

In this whole section, let
\[
f(x)=\Phi_{b}(x)^2 \Phi_{b}(2x-c_b)
 {\rm e}^{-\pi {\rm i} (2x-c_b)^2}
\]
denote the integrand of~\eqref{Psi237}, which is a meromorphic function
with poles at $x_{m/2,n/2}$ for natural numbers $m,n \in \BN$. These poles are
of third order when $m$ and $n$ are even and of first order
otherwise.
First, we compute the contribution from the third order poles $x_{m,n}$.
Using $2 x_{m,n}- c_b = x_{2m,2n}$,~\eqref{eq.phimn},~\eqref{eq.expxmn}
and the modularity transformation of the eta function
\[
\frac{(q;q)_\infty}{(\tq;\tq)_\infty}={\rm e}^{\frac{\pi {\rm i}}{4}}
\left(\frac{\tq}{q}\right)^{\frac{1}{24}} b^{-1},
\]
it follows that
\begin{align*}
 f(x+x_{m,n}) & =  {\rm e}^{-\pi {\rm i} (2x+x_{2m,2n})^2}
 \Phi_b(x+x_{m,n})^2 \Phi_b(2x+x_{2m,2n})
\\
& =  {\rm e}^{\frac{5\pi {\rm i}}{4}} 
\frac{{\rm e}^{{\rm i} \frac{\ve\tv}{\pi}}}{b^3(1-{\rm e}^{-\tv})^2 \big(1-{\rm e}^{-2\tv}\big)} F(q,\ve)
\tilde F(\tq,\tv)
\end{align*}
and
\begin{align*}
& F(q,\ve) =  \frac{(q {\rm e}^\ve;q)_\infty^2 \big(q {\rm e}^{2\ve};q\big)_\infty}{(q;q)_\infty^3}
\sum_{m=0}^\infty
\frac{q^{m(2m+1)} {\rm e}^{\ve (4m+1)}}{
(q {\rm e}^\ve;q)^2_m \big(q {\rm e}^{2\ve};q\big)_{2m}}, \\
& \ti F(q,\ve) =
\frac{(q;q)_\infty^3}{(q {\rm e}^{-\ve};q)_\infty^2\big(q {\rm e}^{-2\ve};q\big)_\infty}
\sum_{n=0}^\infty
\frac{q^{n(n+1)} {\rm e}^{\ve (2n-1)}}{
(q {\rm e}^\ve;q)^2_n \big(q {\rm e}^{2\ve};q\big)_{2n}} .
\end{align*}
It follows that the third order poles contribute
\[
2 \pi {\rm i} {\rm e}^{\frac{5\pi {\rm i}}{4}} 
\mathrm{Res}_{x=0}
\biggl(
\frac{{\rm e}^{{\rm i} \frac{\ve\tv}{\pi}}}{b^3(1-{\rm e}^{-\tv})^2 \big(1-{\rm e}^{-2\tv}\big)} F(q,\ve)
\tilde F(\tq,\tv)
\biggr)
\]
to the state integral. Expanding out, it follows that the contribution
to the state integral is given by
\[
{\rm e}^{\frac{-\pi {\rm i}}{4}} 
\biggl(* \tau H^+_0(q) H^-_2(\tq) + * H^+_1(q) H^-_1(\tq)
+ \frac{*}{\tau} H^+_2(q) H^-_0(\tq) + \frac{*}{2\pi {\rm i}} H^+_0(q) H^-_0(\tq)
\biggr),
\]
where $*$ are easily computable rational numbers.

Next, we compute the contribution from the first order poles. Recall that
\[
f(x+x_{m/2,n/2})  =  {\rm e}^{-\pi {\rm i} (2x+x_{m,n})^2} \Phi_b(x+x_{m/2,n/2})^2
\Phi_b(2x+x_{m,n}).
\]
When $(m,n)$ are not both even,
${\rm e}^{-\pi {\rm i} (2x+x_{m,n})^2} \Phi_b(x+x_{m/2,n/2})^2$ is regular at $x=0$ and
$\Phi_b(2x+x_{m,n})$ has a first order pole at $x=0$. Note that
$x_{(2m+m')/2,(2n+n')/2} = x_{m,n}+ {\rm i} m' b/2 + {\rm i} n' b^{-1}/2$ for $(m',n')=(1,0)$,
$(0,1)$ and $(1,1)$. Equations~\eqref{eq.Phixmn} and~\eqref{eq.phicb}
together with the replacement of $x$ by ${\rm i} m' b/2 + {\rm i} n' b^{-1}/2$
imply that when $(m',n')=(1,0)$,  we have
\begin{gather*}
{\rm e}^{-\pi {\rm i} (x_{(2m+1)/2,2n/2})^2} \Phi_b\big(x_{(2m+m')/2,(2n+n')/2}\big)^2 \\
 \qquad {}=  - {\rm e}^{\frac{\pi {\rm i}}{2}} \biggl(\frac{q}{\tq}\biggr)^{\frac{1}{8}}
q^{(2m+1)(m+1)} \frac{\big(q^{3/2};q\big)^2_\infty}{\big(q^{3/2};q\big)^2_m}
\frac{\tq^{-n^2}}{(-1;\tq)^2_\infty(-\tq;\tq)^2_n}
\end{gather*}
and
\[
\Res_{x=0} \Phi_b(2x+x_{2m+1,2n}) = * {\rm e}^{\frac{3\pi {\rm i}}{4}}
\biggl(\frac{\tq}{q}\biggr)^{\frac{1}{24}} \frac{\tq^{n(2n+1)}}{(q;q)_{2m+1}
(\tq;\tq)_{2n}},
\]
where $*$ is a constant independent of $b$. Likewise, we can treat the case
of $(m',n')=(0,1)$ and~$(1,1)$. With the definition of the six $q$-series
(inside and outside the unit circle) given below, and whose first few
terms are given in~\eqref{eq.237H}, the above computation concludes the
proof of Proposition~\ref{prop.237}.

The series $H^+_k(q)$ and $H^-_k(q)$ for $|q| < 1$ and $k=0,1,2$ are
defined, respectively, by
\begin{gather*}
\frac{(q {\rm e}^\ve;q)_\infty^2 \big(q {\rm e}^{2\ve};q\big)_\infty}{(q;q)_\infty^3}
\sum_{m=0}^\infty
\frac{q^{m(2m+1)} {\rm e}^{\ve (4m+1)}}{
(q {\rm e}^\ve;q)^2_m \big(q {\rm e}^{2\ve};q\big)_{2m}}\\
\qquad{}
=  H^+_0(q) + \ve H^+_1(q) + \frac{\ve^2}{2}
\biggl(H^+_2(q)
+\frac{1}{3} \calE_2(q)H^+_0(q)\biggr)
+ O(\ve)^3 ,
\\
\frac{(q;q)_\infty^3}{(q {\rm e}^{-\ve};q)_\infty^2\big(q {\rm e}^{-2\ve};q\big)_\infty}
\sum_{n=0}^\infty
\frac{q^{n(n+1)} {\rm e}^{\ve (2n+1)}}{
(q {\rm e}^\ve;q)^2_n \big(q {\rm e}^{2\ve};q\big)_{2n}}\\
\qquad{}
=  H^-_0(q) + \ve H^-_1(q) + \frac{\ve^2}{2}
\biggl(H^-_2(q)
+\bigg(\frac{1}{2}-\frac{1}{3} \calE_2(q)\bigg) H^-_0(q)\biggr)
+ O(\ve)^3 ,
\end{gather*}
where
\[
H^+_j(q) = \sum_{m=0}^\infty t_m(q) p_m^{(j)}(q) , \qquad
H^-_j(q) = \sum_{m=0}^\infty T_m(q) P_m^{(j)}(q),
\qquad j=0, 1, 2
\]
with
\begin{gather*}
t_m(q)  =  \frac{q^{m(2m+1)}}{(q;q)^2_m (q;q)_{2m}},
\qquad
T_n(q)  =  \frac{q^{n(n+1)}}{(q;q)^2_n (q;q)_{2n}} , 
\end{gather*}
(not to be confused with~\eqref{tT52}) and
\begin{gather*}
  p_m^{(0)}(q) =1 , \qquad p_m^{(1)}(q) = 4m+1-2E_1^{(m)}(q)-2\calE_1^{(2m)}(q)
 , \\
p_m^{(2)}(q) = p_m^{(1)}(q)^2 - 2 \calE_2^{(m)}(q) - 4 \calE_2^{(2m)}(q)
- \frac{1}{3} \calE_2(q) , \\
  P_n^{(0)}(q) =1 , \qquad P_n^{(1)}(q) = 2n+1-2 \calE_1^{(n)}(q) -2 \calE_1^{(2n)}(q)
 , \\
 P_n^{(2)}(q) = P_n^{(1)}(q)^2 +12 \calE_2^{(0)}-\frac{1}{2}
- 2 \calE_2^{(n)}(q) - 4 \calE_2^{(2n)}(q) + \frac{1}{3} \calE_2(q) .
\end{gather*}
The remaining series $H^+_k(q)$ and $H^-_k(q)$ for $|q| < 1$ and $k=4,5,6$ are
defined, respectively, by
\begin{gather*}
H^+_3(q)    =  \frac{\bigl(q^{3/2};q\bigr)_\infty^2}{(q;q)_\infty^2}
\sum_{m=0}^\infty \frac{q^{(2m+1)(m+1)}}{\bigl(q^{3/2};q\bigr)_m^2 (q;q)_{2m+1}},
\\
  H^-_3(q)    =  \frac{(q;q)_\infty^2}{\bigl(q^{-1/2};q\bigr)_\infty^2}
\sum_{n=0}^\infty \frac{q^{n(n+2)}}{\bigl(q^{3/2};q\bigr)_n^2 (q;q)_{2n+1}},
\\
H^+_4(q)    =  \frac{(-q;q)_\infty^2}{(q;q)_\infty^2}
\sum_{m=0}^\infty \frac{q^{(2m+1)m}}{(-q;q)_m^2 (q;q)_{2m}},
\\
  H^-_4(q)    =  \frac{(q;q)_\infty^2}{(-1;q)_\infty^2}
\sum_{n=0}^\infty \frac{q^{n(n+1)}}{(-q;q)_n^2 (q;q)_{2n}},
\\
H^+_5(q)    =  \frac{\bigl(-q^{3/2};q\bigr)_\infty^2}{(q;q)_\infty^2}
\sum_{m=0}^\infty \frac{q^{(2m+1)(m+1)}}{\bigl(-q^{3/2};q\bigr)_m^2 (q;q)_{2m+1}},
\\
  H^-_5(q)    =  \frac{(q;q)_\infty^2}{\bigl(-q^{-1/2};q\bigr)_\infty^2}
\sum_{n=0}^\infty \frac{q^{n(n+2)}}{\bigl(-q^{3/2};q\bigr)_n^2 (q;q)_{2n+1}}.
\end{gather*}

Note that the above $q$-hypergeometric series are convergent
for $|q| \neq 1$ and satisfy the symmetries
\begin{alignat*}{4}
\label{eq.6sym1}
& H^{+}_0\big(q^{-1}\big)  = H^{-}_0(q),\qquad  && H^{+}_1\big(q^{-1}\big)  = -H^{-}_1(q),\qquad  && H^{+}_2\big(q^{-1}\big)
 = H^{-}_2(q), &
\\
& H^{+}_3\big(q^{-1}\big)  = -H^{-}_4(q),\qquad  && H^{+}_4\big(q^{-1}\big)  = H^{-}_4(q),\qquad &&  H^{+}_5\big(q^{-1}\big)
 = -H^{-}_5(q) . &
\end{alignat*}
\begin{Remark}
Despite appearances, $H^+_0(q) (q;q)_\infty^2$ \big(as well as
$H^-_0(q) (q;q)_\infty^2$ as the other 10 $q$-series\big)
is a rank 3 Nahm sum. Indeed, use
\[
(q;q)_\infty^2 T_m(q) = \frac{q^{m(2m+1)}}{(q;q)_m}
\big(q^{m+1};q\big)_\infty \big(q^{2m+1};q\big)_\infty
\]
together with the identity
\begin{gather}
\label{eq.qxinf}
(qx;q)_\infty  =
\sum_{k=0}^\infty (-1)^k \frac{q^{\frac{k(k+1)}{2}}x^k}{(q;q)_k}
\end{gather}
to obtain that
\[
(q;q)_\infty^2 H^+_0(q)  =  \sum_{k,l,m} (-1)^{k+l}
\frac{q^{\frac{1}{2}(k^2+l^2+4m^2+2km+4lm)+\frac{1}{2}(k+l+2m)}}{
(q;q)_k (q;q)_l (q;q)_m} .
\]
\end{Remark}

\subsection*{Acknowledgements}

The authors would like to thank Tudor Dimofte, Gie~Gu, Rinat Kashaev,
Marcos Mari\~no and Sander Zwegers for their input and the anonymous referees
for a careful reading of the manuscript.

\pdfbookmark[1]{References}{ref}
\LastPageEnding

\end{document}